\documentclass[a4paper,11pt]{amsart}
\usepackage[utf8]{inputenc}
\usepackage[T1]{fontenc}
\usepackage{amsmath,amsthm,amssymb}
\usepackage{amsfonts}
\usepackage{mathtools}
\usepackage{bbm}
\usepackage{stmaryrd}
\usepackage{mathrsfs}
\usepackage{caption, subcaption}
\usepackage{enumitem}
\usepackage{tikz}\usetikzlibrary{decorations.markings, patterns, calc}
\usepackage{hyperref}
\usepackage{xcolor}
\usepackage{doi}

\numberwithin{equation}{section}

\newcommand{\C}{{\mathbb C}}
\newcommand{\N}{{\mathbb N}}
\newcommand{\Z}{{\mathbb Z}}
\newcommand{\R}{{\mathbb R}}
\newcommand{\Cal}{\mathcal}
\newcommand{\mfrac}[2]{\tfrac{#1\strut}{#2\strut}}

\newcommand{\sminus}{\setminus} 
\newcommand{\hot}{\mathrm{h.o.t.}}

\newcommand{\re}{\operatorname{Re}}
\renewcommand{\Re}{\operatorname{Re}}
\newcommand{\im}{\operatorname{Im}}

\renewcommand{\i}{\mathrm{i}\mkern.25\thinmuskip}
\renewcommand{\d}{\ensuremath{\operatorname{d}\!}}
\newcommand{\e}{\ensuremath{\mkern.25\thinmuskip\mathrm{e}\mkern.4\thinmuskip}}
\newcommand{\res}{\operatorname{res}}

\newcommand{\tdd}[1]{\tfrac{\partial}{\partial #1 \vphantom{\hat X}}}
\newcommand{\model}{{\mathrm{mod}}}
\newcommand{\fmod}{f_\model}

\newcommand{\psimod}{\psi_\model}
\newcommand{\ximod}{\xi_\model}
\newcommand{\orbit}{\mathbb{A}}
\newcommand{\orbitmod}{{\orbit_\model}}
\newcommand{\Aa}{\mathfrak{A}}
\newcommand{\Borel}{\mathscr{B}}
\newcommand{\Laplace}{\mathscr{L}}
\newcommand{\BSum}{\mathscr{S}}
\newcommand{\BV}{\operatorname{BV}}
\newcommand{\Sing}{\operatorname{Sing}}

\newcommand{\Circ}{\operatorname{Circ}}
\newcommand{\VP}{\operatorname{p.\!v.\!}}
\newcommand{\dist}{\operatorname{dist}}
\newcommand{\pls}{{\scalebox{0.7}{$\pmb{+}$}}}
\newcommand{\mns}{{\scalebox{0.7}{$\pmb{-}$}}}
\newcommand{\pms}{{\scalebox{0.7}{$\pmb{\pm}$}}}

\newtheorem{theorem}{Theorem}

\newtheorem{proposition}[theorem]{Proposition}

\newtheorem{thmx}{Theorem}
\newtheorem{cx}{Corollary}

\theoremstyle{definition}

\newtheorem{notation}[theorem]{Notation}
\newtheorem{remark}[theorem]{Remark}

\def\sideremark#1{\ifvmode\leavevmode\fi\vadjust{\vbox to0pt{\vss 
			\hbox to 0pt{\hskip\hsize\hskip1em           
				\vbox{\hsize3.5cm\tiny\raggedright\pretolerance10000
					\noindent #1\hfill}\hss}\vbox to8pt{\vfil}\vss}}}%

\title[Reading analytic invariants of orbits]{Reading analytic invariants\\ of parabolic diffeomorphisms\\ from their orbits}

\author[M. Klime\v{s}]{Martin Klime\v{s}}
\address{University of Zagreb, Faculty of Electrical Engineering and Computing, Unska 3, 10000 Zagreb, Croatia.}
\email{\href{mailto:martin.klimes@fer.hr}{martin.klimes@fer.hr}}

\author[P. Marde\v{s}i\'c]{Pavao Marde\v{s}i\'c}
\address{Universit\'e de Bourgogne, Institut de Mathématiques de Bourgogne, 9 avenue Alain Savary, 21078 Dijon, France.
and University of Zagreb, Faculty of Science, Department of Mathematics, Bijenička cesta 30, 10000 Zagreb, Croatia.
}
\email{\href{mailto:pavao.mardesic@u-bourgogne.fr}{pavao.mardesic@u-bourgogne.fr}}

\author[G. Radunovi\'c]{Goran Radunovi\'c}
\address{University of Zagreb, Faculty of Science, Department of Mathematics, Bijenička cesta 30, 10000 Zagreb, Croatia.}
\email{\href{mailto:goran.radunovic@math.hr}{goran.radunovic@math.hr}}

\author[M. Resman]{Maja Resman}
\address{University of Zagreb, Faculty of Science, Department of Mathematics, Bijenička cesta 30, 10000 Zagreb, Croatia.}
\email{\href{mailto:maja.resman@math.hr}{maja.resman@math.hr}}

\begin{document}


\thanks{All authors are  supported by the Croatian Science Foundation grant PZS-2019-02-3055 and the bilateral Hubert-Curien Cogito grant 2021-22. Goran Radunovi\' c and Maja Resman are supported also by the Croatian Science Foundation grant no. UIP-2017-05-1020. Pavao Marde\v si\' c was partially supported by  
EIPHI Graduate School (contract ANR-17-EURE-0002). Pavao Marde\v si\' c and Maja Resman also express their gratitude to the Fields Institute for supporting their stay in the scope of the \emph{Thematic Program on Tame Geometry, Transseries and Applications to Analysis and Geometry 2022}.}

\begin{abstract}
In this paper we study germs of diffeomorphisms in the complex plane. We address the following problem: \emph{How to read a diffeomorphism $f$ knowing one of its orbits $\mathbb{A}$?} 

We solve this problem for parabolic germs. 
This is done by associating to the orbit ${\mathbb{A}}$ a function that we call the \emph{dynamic theta function} $\Theta_{\mathbb{A}}$. We prove that the function $\Theta_{\mathbb{A}}$ is $2\pi\i\mathbb{Z}$-resurgent.
We show that one can obtain the sectorial Fatou coordinate as a Laplace-type integral transform of the function $\Theta_{\mathbb{A}}$.
This enables one to read the analytic invariants of a diffeomorphism from the theta function of one of its orbits. 

We also define a closely related \emph{fractal theta function} $\tilde{\Theta}_\orbit$,  
which is inspired by and generalizes  the \emph{geometric zeta function} of a fractal string, and show that it also  encodes the analytic invariants of the diffeomorphism.  
\end{abstract}

\maketitle

{\small \tableofcontents}

\section{Introduction}
Consider a germ of an analytic diffeomorphism of $(\C,0)$ with a parabolic fixed point at the origin
\begin{equation}\label{eq:f}
	f(x)=x+ax^{k+1}+\hot(x), \qquad a\neq 0,\quad k>0.
\end{equation}
By the Leau--Fatou theorem (see \cite{Bracci,IlYa,Loray}), given a disc $D=\{x\in\C,\ |x|<r\}$ of sufficiently small radius $r>0$,  
the set of points \begin{equation}\label{eq:Df}
D_f=\bigcap_{n\in\N}f^{\circ(-n)}(D)
\end{equation}
whose all positive iterates stay in $D$, consists of $k$ connected petals, disjoint except at the origin, centered around the $k$ rays $\{-ax^k\in\R_{\geq 0}\}$. Similarly for $D_{f^{\circ(-1)}}=\bigcap_{n\in\N}f^{\circ n}(D)$, the petals are centered around the rays $\{ax^k\in\R_{\geq 0}\}$. 
In fact, by a theorem of Camacho and Shcherbakov (see \cite{Bracci,IlYa,Loray}), $f$ is topologically conjugated by a quasiconformal tangent-to-identity map to the germ
\begin{equation}\label{eq:fmod}
	\fmod(x)=\frac{x}{(1-akx^k)^{\frac1k}}=\exp\big(ax^{k+1}\tdd{x}\big)(x),
\end{equation}
given by the time-1-flow of the model vector field $\ximod=ax^{k+1}\tdd{x}$.

Let $0\neq x_0\in D_f$ be a point in one of the attractive petals, hence $f^{\circ n}(x_0)\to 0$ as $n\to+\infty$, and let
\begin{equation}\label{eq:orbit}
	\orbit=\{x_n=f^{\circ n}(x_0),\ n\in\N\}\qquad(\text{for us}\ \N:=\Z_{\geq 0})
\end{equation} 
be its \emph{forward orbit} by $f$.
A basic question is \emph{how much information does one such orbit $\orbit$ carry about $f$?}
The answer is trivial: 
Since $f$ is an analytic germ at $0$, whose values on $\orbit$ are determined by the relation $f(x_n)=x_{n+1}$,
and $\orbit$ accumulates at $0$, this means that \emph{$\orbit$ determines $f$ as a germ}. 

Instead, \emph{our question is whether one can somehow ``read'' the analytic invariants of $f$} (i.e. the analytic conjugacy class of $f$) \emph{``directly'' from the discrete set $\orbit$?} 
That is, knowing a set $\orbit$ and knowing that it is an orbit of some (unknown) analytic diffeomorphism, can one determine the analytic invariants of this diffeomorphism.

The orbit $\orbit=\{x_0,x_1,x_2,\ldots\}$ can be considered either as a subset of some neighborhood of the origin in $\C$ or as a sequence of distinct points. 
We shall alternate freely between the two points of view.

In order to read analytic properties of parabolic diffeomorphisms from its orbits, we introduce two theta functions: a \emph{dynamic theta function $\Theta_\orbit$} and a \emph{fractal theta function $\tilde\Theta_\orbit$}. We prove two main theorems: Theorem \ref{thm:1} and Theorem \ref{cor:1} showing the properties of the dynamic theta function $\Theta_\orbit$. In Proposition \ref{prop:tildeTheta}, we show the relationship between the two theta functions, hence transforming Theorem \ref{thm:1} and Theorem \ref{cor:1} to  Corollaries \ref{thm:2} and \ref{cor:2} giving the analogous results for the fractal theta function $\tilde\Theta_\orbit$.

\subsection{Dynamic theta function of an orbit.}
The \emph{dynamic theta function $\Theta_\orbit$ of an orbit $\orbit$} is defined by
\begin{equation}\label{eq:Theta}
	\Theta_\orbit(s):=\sum_{x_i\in\orbit}\,\e^{-s\,t(x_i)}, 
\end{equation}	
where
\begin{equation}\label{eq:psimod}
	 	t(x)=\psimod(x)=-\tfrac{1}{ak}x^{-k}
\end{equation}	
is the Fatou coordinate of the model \eqref{eq:fmod}.
The function $\Theta_\orbit$ is well-defined and analytic for $\re(s)>0$, and we prove in Theorem \ref{thm:1} that it can be indefinitely analytically continued as a multivalued function. The function $\Theta_\orbit$ for  real orbits  is also known as \emph{geometric partition function} of the fractal string $\big\{t(x_i)^{-1}\big\}_{i=1}^{\infty}$ \cite{LapFr2}. The definition of  $\Theta_\orbit(s)$ is also similar to that of a \emph{spectral theta function} of Voros \cite{Voros1,Voros2}, thus the name. 

Denoting $\delta_{x_i}(x)$ the Dirac distribution supported at $x_i$, let  \[\delta_\orbit(x):=\sum_{x_i\in\orbit}\delta_{x_i}(x).\]
Then $\Theta_\orbit(s)$ defined in \eqref{eq:Theta} is a Laplace-type integral transform of $\delta_\orbit(x)$ with weight $t(x)$: 
\begin{equation}\label{eq:Thetadelta}
	\Theta_\orbit(s)=\int_{\gamma}\delta_\orbit(x)\, \e^{-s\,t(x)}\d x,
\end{equation}
where $\gamma\subset\C$ is any curve through the points of $\orbit$ in their order.

In particular, for an orbit $\orbitmod$ of the model $\fmod$, one has
\[\Theta_\orbitmod(s)=\sum_{n=0}^{+\infty}\e^{-s\,(t(x_0)+n)}=\frac{\e^{-s\,t(x_0)}}{1-\e^{-s}},\]
which is a meromorphic function with simple poles at $2\pi\i\Z$.

The idea behind the function $\Theta_\orbit(s)$ is to mimic the method of analysis of singularities of the Borel transform of a formal Fatou coordinate  of $f$.
This was initiated in \'Ecalles' pioneering works on analytic classification of parabolic diffeomorphisms \cite{Ecalle,Ecalle1}, and was a starting point in his development of theory of resurgence \cite{Ecalle2}.

The orbit $\orbit$ is contained in some sector $W_\orbit$ contained in one of the attractive petals of $D_f$ \eqref{eq:Df}.
By the Leau--Fatou theorem (Theorem~\ref{thm:sectorialFatou} in \S\ref{sec:3.1}), on $W_\orbit$ there exists a sectorial Fatou coordinate $\psi_\orbit(x)$, 
i.e. mapping verifying $\psi_\orbit\circ f=\psi_\orbit+1$. It is uniquely
determined by the initial condition $\psi_\orbit(x_0)=0.$
Then $\psi_\orbit(x_m)=m$ for all $m\in\N$. 
We denote $\hat\psi_\orbit(x)$ the formal Fatou coordinate to which $\psi_\orbit$ is asymptotic on $W_\orbit$.

\medskip
The first of our principal results is the following:

 \begin{thmx}\label{thm:1}
The \emph{dynamic theta function} $\Theta_\orbit(s)$ extends analytically to the universal cover of $\C\setminus 2\pi\i\Z$ and has at most exponential growth at infinity along any non-vertical direction.
 	
The unique sectorial Fatou coordinate $\psi_\orbit(x)$ on the petal containing the orbit $\orbit$ such that $\psi_\orbit(x_0)=0$
can be obtained from $\Theta_\orbit(s)$ through the integral transform 
  	\begin{equation}\label{eq:ThetaFatou}
 		\psi_\orbit(x)=\tfrac{1}{2\pi\i}\int_{\Circ(\R_{\leq 0})}\tfrac{\Theta_\orbit(s)}{s}\,\e^{s\,t(x)}\d s=:\widetilde\Laplace\big\{\tfrac{\Theta_\orbit(s)}{2\pi\i\,s}\big\}(x),
 	\end{equation}
 where $\Circ(\R_{\leq 0})$ is the Hankel contour encircling the ray $\R_{\leq 0}$ (Figure~\ref{figure:1}).
\end{thmx}

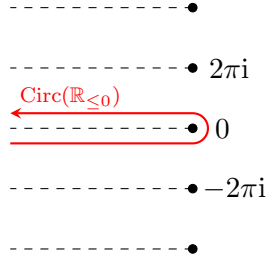
\begin{figure}
\centering
\begin{tikzpicture}[scale=.8]
	\foreach \n in {-2,...,2} {
		\begin{scope}
			\draw[dashed] (0,\n) -- +(180:3);
			\filldraw (0,\n) circle (2pt); 
		\end{scope}
	};	
	\draw[-stealth, thick,red] (-3,-0.25) -- (0,-0.25) arc[start angle=-90, end angle=90,radius=0.25] -- (-3,.25);			
	\draw (0.2,0) node[right]{$0$};		
	\draw (0.1,1) node[right]{$2\pi\i$};
	\draw (0,-1) node[right]{$-2\pi\i$};				
	\draw[red] (-2,0.15) node[above]{$_{\Circ(\R_{\leq 0})}$};
\end{tikzpicture}
\caption{The Hankel contour $\Circ(\R_{\leq 0})$ in the $s$-plane. Dashed lines indicate ramification cuts of the multivalued function $\Theta_\orbit(s)$.}
\label{figure:1}
\end{figure}

The formula \eqref{eq:ThetaFatou} means that one can recover the sectorial Fatou coordinate $\psi_\orbit(x)$, and thus the germ 
\[
f(x)=\psi_\orbit^{\circ(-1)}\circ(\psi_\orbit(x)+1),
\]
	from the knowledge of a single orbit $\orbit$.
	One can summarize the process of reconstruction of $f$ from an orbit $\orbit$ by the following diagram:
\[
\orbit\mapsto \Theta_\orbit\xmapsto{\widetilde\Laplace\{\frac{\cdot}{2\pi\i\,s}\}}\psi_\orbit\mapsto f.
\]

One can interpret \eqref{eq:ThetaFatou} as a discrete version of the usual Borel summation process.
Let $\hat\psi_\orbit(x)$ be the formal Fatou coordinate for $f$, to which is $\psi_\orbit(x)$ asymptotic. Let $\tilde\Psi_\orbit(s)$ be its
formal Borel transform defined term-wise, 
\[\tilde\Psi_\orbit(s)=\tfrac{1}{2\pi\i}\int_{t(x)\in t(x_0)+\R_{\geq 0}}\hat\psi_\orbit(x)\,\e^{-s\,t(x)}\d t(x) =:\widetilde\Borel\{\hat\psi_\orbit\}(s).\]
Then, by the Borel-Laplace summation method,
\begin{equation}\label{eq:orb}\psi_\orbit(x)=\int_{\Circ(\R_{\leq 0})}\tilde\Psi_\orbit(s)\,\e^{s\,t(x)}\d s =:\widetilde\Laplace\{\tilde\Psi_\orbit\}(x)\end{equation}
is a sectorial Fatou coordinate for $f$ on some sector $W_\orbit$ containing the orbit $\orbit$.  
The identity between the integrals \eqref{eq:ThetaFatou} and \eqref{eq:orb} is due to the difference $\frac{\Theta_\orbit(s)}{2\pi\i\,s}-\Tilde\Psi(s)$ being analytic on a neighborhood of the ray $\R_{\leq 0}$.
More details about the relation between the function $\Theta_\orbit(s)$ and the Borel transform of the formal Fatou coordinate in the Borel plane are given in Proposition~\ref{prop:Theta}.

\begin{thmx}\label{cor:1}
The Birkhoff--\'Ecalle--Voronin invariants of the germ $f$ 
can be read from the leading term coefficients at the singularities of the analytic continuation of $\Theta_\orbit(s)$.
\end{thmx}
\begin{proof} 
The analytic class of $f$, given by so called Birkhoff--\'Ecalle--Voronin invariants (see \S\ref{sec:3.1}),
can be read from the function $\tilde\Psi_\orbit(s)$ in terms of the leading term coefficients at some of its singularities (see \S\ref{sec:3.3}). The same then holds for $\Theta_\orbit(s)$, by formula \eqref{eq:BVThetaFatou} in Proposition~\ref{prop:Theta}. 
\end{proof}

\subsection{Fractal analysis: fractal theta function of an orbit.}\label{sub:fraan}\
Another way to analyze the orbit $\orbit$ is by investigating its fractal properties.
In this subsection,  we introduce a function $\tilde\Theta_\orbit$,
which we call \emph{fractal theta function}, as an integral transform of the \emph{counting function $n_\orbit$ of the orbit}, which encodes fractal properties of $\orbit$.
This is similar to the fractal zeta function defined by Lapidus \cite{LRZ, LapFr1} and studied in \cite{MRR}, but with an integral kernel well adapted to reveal intrinsic properties of parabolic dynamics.
The fractal theta function $\tilde\Theta_\orbit$ is related to the dynamic theta function $\Theta_\orbit$ of the orbit $\orbit$ of the previous section.
As a consequence, it is possible to read the analytic invariants of a germ $f$ from fractal properties of its orbit.
\smallskip

In this section, let us assume that the orbit $\orbit$ \eqref{eq:orbit} 
lies in the positive real axis:
\[\orbit=\{x_0>x_1>x_2>\ldots\}\subset\R_{>0}.\] 
This means that 
$f:(\R_{\geq 0},0)\to(\R_{\geq 0},0)$ is real analytic with $a<0$. 
For each $\epsilon>0$, let 
\[\orbit_{\epsilon}=\{x\in\R,\ \dist(x,\orbit)<\epsilon\}\]
be the $\epsilon$-neighborhood of the orbit $\orbit$ in $\R$,
and denote by $\big|\orbit_\epsilon\big|$ its Lebesque measure. 
The \emph{tube function}
\[V_\orbit(\epsilon): \ \epsilon\mapsto \big|\orbit_\epsilon\big|-2\epsilon=\big|\orbit_\epsilon\cap[0,x_0]\big|\]
encodes the fractal properties of the orbit $\orbit$.
In particular, if $V_\orbit(\epsilon)=\mathcal{M}\,\epsilon^{1-D}+o(\epsilon^{1-D})$, then $D$ is the \emph{Minkowski $($box$\,)$ dimension} of $\orbit$, and $\mathcal{M}$ is the \emph{Minkowski content} of $\orbit$.

The tube function $V_\orbit(\epsilon)$ of forward orbits of parabolic analytic diffeomorphisms has been extensively studied by the authors in \cite{MRR,MRRZ1,Resman1,Resman2}.  
In particular, it has been determined \cite{EZZ,MRR} that 
\begin{equation}\label{eq:Minkowski}
	D=\tfrac{k}{k+1},\quad\text{and } \ \mathcal{M}=\tfrac{k+1}{k}\big(\!-\!\tfrac{2}{a}\big)^{\frac1{k+1}}.
\end{equation}
Moreover, it has been shown in \cite{Resman2} that an asymptotic expansion of $V_\orbit(\epsilon)$ as a power-log series exists up to the order $\epsilon^{1+\frac{k}{k+1}}$, and that this finite jet determines the formal class of $f$.
Beyond this order some oscillatory terms appear, and a full asymptotic expansion containing these oscillatory terms has been given in \cite{MRR}.

Instead of looking at the tube function $V_\orbit(\epsilon)$, it is more convenient to look at the closely related \emph{counting function}
\[n_\orbit(\epsilon):=\max\{n\in\N_{>0}:\ |x_{n-1}-x_n|\geq2\epsilon\},\]
counting the maximal number of iterates $n>0$ such that the $\epsilon$-neighborhoods of the points $x_0,\ldots,x_{n}$ are disjoint.
For each $\epsilon>0$, the orbit $\orbit$ is divided into two parts: 
the \emph{tail} $T_\orbit(\epsilon)=\{x_0,\ldots,x_{n_\orbit(\epsilon)-1}\}$,  whose $\epsilon$-neighborhood consist of $n_\orbit(\epsilon)$ disjoint intervals of length $2\epsilon$, and
the \emph{nucleus} $N_\orbit(\epsilon)=\{x_j: j\geq n_\orbit(\epsilon)\}$, whose $\epsilon$-neighborhood spans the interval $]-\epsilon,x_{n_\orbit(\epsilon)}+\epsilon\,[$.
Hence the tube function is expressed in terms of the counting function as
\begin{equation}\label{eq:VA}
	V_\orbit(\epsilon)=2\epsilon\, n_\orbit(\epsilon)+x_{n_\orbit(\epsilon)}.
\end{equation}

Denoting
\begin{equation}\label{eq:eps}\epsilon_n=\tfrac12(x_{n-1}-x_n)=\tfrac12 g(x_{n-1}),\qquad g(x)=x-f(x),\end{equation}
we have:
\[n_\orbit(\epsilon)=\# \{n\in\N_{>0}:\ \epsilon_n\geq\epsilon\}=\sum_{n=1}^{+\infty}\mathbbm{1}_{]0,\epsilon_n]}(\epsilon),\]
where $\mathbbm{1}_{]0,\epsilon_n]}(\epsilon)$ is the characteristic function of the interval $]0,\epsilon_n]$.

\begin{proposition}\label{prop:tubefunction}
Let $\orbit\subset\R_{>0}$. The tube function $V_\orbit(\epsilon)$ is an integral of $2n_\orbit(\epsilon)$:	
	\begin{equation}\label{eq:VnA}
		V_\orbit(\epsilon)=2\int_0^\epsilon n_\orbit(\xi)\d\xi,\qquad \epsilon\geq 0.	
	\end{equation}
\end{proposition}
\begin{proof}
By \eqref{eq:VA}, the function $V_\orbit(\epsilon)$ is continuous, piece-wise linear, with the derivative $\frac{\d}{\d\epsilon}V_\orbit(\epsilon)$ equal to $2n_\orbit(\epsilon)=2n$ on each interval $\epsilon\in\ ]\epsilon_{n+1},\epsilon_n[$,	$n\in\N_{>0}$, and $V_\orbit(0)=0$. 
\end{proof}

We define the \emph{fractal theta function} $\tilde{\Theta}_\orbit(s)$ by
\begin{equation}\label{eq:tildeTheta}
\tilde{\Theta}_\orbit(s)=\sum_{n=1}^{+\infty}\e^{-s\tau(\epsilon_n)},\qquad \re(s)>0,
\end{equation}
where 
\begin{equation}\label{eq:tau}
	\tau(\epsilon)=\tfrac{1}{2k}\big(\!-\!\tfrac{2}{a}\big)^{\frac1{k+1}}\epsilon^{-\frac{k}{k+1}}
\end{equation}
is the leading asymptotic term of $n_\orbit(\epsilon)$ (cf. \eqref{eq:Minkowski} and \eqref{eq:VnA}).
Equivalently, it is equal to the following integral transform of the counting function
\begin{equation*}
		\tilde{\Theta}_\orbit(s)=s\int_{\tau(\epsilon)\in\R_{>0}}\!\!\!\!n_\orbit(\epsilon)\,\e^{-s\,\tau(\epsilon)}\d\tau(\epsilon)=
		\sum_{n=1}^{+\infty}s\int_{\tau(\epsilon)\in\,[\tau(\epsilon_n),+\infty[}\!\!\!\e^{-s\,\tau(\epsilon)}\d\tau(\epsilon).
\end{equation*}

Using Proposition~\ref{prop:tubefunction} and integration by parts, one can also rewrite  $\tilde\Theta_\orbit(s)$ in terms of the tube function $V_\orbit(\epsilon)$:
\begin{equation*}\label{eq:ftf}
\begin{aligned}
\tilde{\Theta}_\orbit(s)&=-\tfrac12 \int_0^{\epsilon_1} \tfrac{\d}{\d\epsilon} V_\orbit(\epsilon)\,\tfrac{\d}{\d\epsilon}\e^{-s\,\tau(\epsilon)}\d\epsilon\\
&=\tfrac12 \int_0^{\epsilon_1} V_\orbit(\epsilon)\,\tfrac{\d^2}{\d\epsilon^2}\e^{-s\,\tau(\epsilon)}\d\epsilon
	\, + \,\tfrac{sx_0}{2}\, \e^{-s\,\tau(\epsilon_1)}{\tau}'(\epsilon_1). 
\end{aligned}	
\end{equation*}
Another way to rewrite it is in terms of the distance function, 
which maps each of the intervals $[x_j,\frac{x_j+x_{j-1}}{2}]$ and $[\frac{x_j+x_{j-1}}{2},x_{j-1}]$ bijectively to
$[0,\epsilon_j]$, $j\in\N_{>0}$, 
\begin{equation}
\tilde{\Theta}_\orbit(s)=-\tfrac{s}{2}\int_0^{x_0} \e^{-s\,\tau(\dist(x,\orbit))} {\tau}'(\dist(x,\orbit))\d x.	
\end{equation}

\begin{remark}
One can define the function $\tilde\Theta_\orbit$ in a similar way also for hyperbolic germs $f(x)=\e^{-\alpha}x+\hot(x)$, $\re(\alpha)>0$, with model
\[f_\model(x)=\e^{-\alpha}x=\exp(-\alpha x\tdd{x}),\qquad \psimod(x)=-\tfrac1\alpha\log x.\] 
Then for $\,\tau(\epsilon)=-\tfrac1\alpha\log\epsilon\,$ one has \
$\tilde\Theta_\orbit(s)=\sum_{n=1}^{+\infty}\epsilon_n^{\frac{s}{\alpha}}
=\zeta_{\Cal E_\orbit}(\tfrac{s}{\alpha})$, 
where the function 
\[\zeta_{\Cal E_\orbit}(s)=\sum_{n=1}^{+\infty}\epsilon_n^{s}
=s\!\int_{\epsilon\in\R_{>0}}\!\!\!\!\! n_\orbit(\epsilon)\,\epsilon^{s-1}\!\d\epsilon
=\tfrac{s}{2}\!\int_0^{x_0}\!\! \dist(x,\orbit)^{s-1}\!\d x,\qquad \re(s)>0,\] 
is known as the \emph{geometric zeta function of the fractal string} $\Cal E_\orbit=\{\epsilon_n: n\in\N_{>0}\}$.
\end{remark}

Geometric zeta functions of fractal strings have been the object of great interest in fractal analysis, see e.g. \cite{LRZ,LapFr1,LapFr2} and can be considered as a special case of fractal theta functions 
with the choice $\tau(\epsilon)=-\log\epsilon$.
Their application to orbits of parabolic diffeomorphisms has been extensively studied in \cite{MRR} where the formal class of the diffeomorphism was then successfully read from the geometric zeta function.
However, in order to read its analytic class, in this paper, we have to adjust the kernel $\tau$ as in \eqref{eq:tau}.

\begin{remark}\label{rem: complex tilda theta}
The formula \eqref{eq:tildeTheta}
with $\epsilon_n\in\mathbb C$ defined by \eqref{eq:eps}, extends the notion of fractal theta function to any  \emph{complex} orbit $\orbit\subset\C$.                                                                         
One can also extend the notion of the counting function $n_\orbit(\epsilon)$ to complex orbits in the following analytic way:
For any embedding of the complex fractal string $\Cal E_\orbit=\{\epsilon_n: n\in\N_{>0}\}$ into a real curve $\Cal C\subset\C^*$, which passes through the points in their order and approaches $0$ with an asymptotic direction (i.e. without too much oscillation), one defines
$n_\orbit:\Cal C\to\N$ by $n_\orbit(\epsilon)=n$ on each of the curve intervals $\Cal C_{]\epsilon_{n+1},\epsilon_n]}$ between $\epsilon_{n}$ and $\epsilon_{n+1}$, and by $0$ on $\Cal C\sminus\bigcup_{n=1}^{+\infty}\Cal C_{]\epsilon_{n+1},\epsilon_n]}$.
Then clearly $n_\orbit(\epsilon)=\sum_{n=1}^{+\infty}\mathbbm {1}_{]\epsilon_{n+1},\epsilon_n]}(\epsilon)$. 
From \eqref{eq:tildeTheta} and by partial integration we thus get:
\begin{equation}\label{eq:tildeThetacomplex}\tilde{\Theta}_\orbit(s)=\int_{\Cal C}\sum_{n=1}^{+\infty}\delta_{\epsilon_n}(\epsilon)\,\e^{-s\tau(\epsilon)}\d\epsilon=
s\int_{\tau(\epsilon)\in\tau(\Cal C)}n_\orbit(\epsilon)\,\e^{-s\tau(\epsilon)}\d\tau(\epsilon),\end{equation}
regardless of the curve $\Cal C$.
\end{remark}

The fractal theta function is in fact the dynamic theta function of a conjugated object:

\begin{proposition}[Relation between the dynamic and the fractal theta function]\label{prop:tildeTheta}
	Let $\phi(x)=x+\hot(x)$ be the analytic germ such that $\phi(x)^{k+1}=\frac{f(x)-x}{a}$,
	and let 
	\[\tilde\orbit=\{\tilde x_n=\phi(x_n),\ n\in\N\}=\phi(\orbit)\] 
	be the image of  the orbit $\orbit=\{x_n,\ n\in\N\}$ by $\phi$. Then $\tilde\orbit$ is an orbit of the analytically conjugated diffeomorphism
	$\tilde f=\phi\circ f\circ\phi^{\circ(-1)},$
	and
\begin{equation*}\label{eq:thetatildetheta}
	\tilde\Theta_\orbit(s)=\Theta_{\tilde\orbit}(s),
\end{equation*}	
where $\Theta_{\tilde\orbit}(s)$ is the dynamic theta function \eqref{eq:Theta} of the orbit $\tilde\orbit$.
\end{proposition}

\begin{proof}
By comparing \eqref{eq:tau} and \eqref{eq:psimod}, and by definition of $\phi$, we have 
\begin{equation*}\label{eq:phi}
\tau(\epsilon)=t\circ\tilde g^{\circ(-1)}(2\epsilon),\qquad\text{and}\qquad   \phi(x)=\tilde g^{\circ(-1)}\circ g(x), 
\end{equation*}
where $\tilde g=-ax^{k+1}$ and $g(x)=x-f(x)$.
Therefore
\[t(\tilde x_n)=t\circ\phi(x_n)=t\circ \tilde g^{\circ(-1)}\circ g(x_n)=t\circ\tilde g^{\circ(-1)}(2\epsilon_{n+1})=\tau(\epsilon_{n+1}),\]
and	
$\ \tilde\Theta_\orbit(s)=\sum_{n=0}^{+\infty}\e^{-s\,\tau(\epsilon_{n+1})} 
=\sum_{n=0}^{+\infty}\e^{-s\,t(\tilde x_n)}=\Theta_{\tilde\orbit}(s).$
\end{proof}

As a corollary, we have the following analogue of Theorem~\ref{thm:1} for the fractal theta function:

\begin{cx}\label{thm:2}
 	The \emph{fractal theta function} $\tilde\Theta_\orbit(s)$ extends analytically to the universal cover of\  $\C\sminus 2\pi\i\Z$ and has at most exponential growth at infinity along any non-vertical direction.

 The \emph{critical time} function $T_\orbit(\epsilon):=\psi_\orbit\circ g^{\circ(-1)}(2\epsilon)$, that satisfies:
 	\[T_\orbit(\epsilon_n)=n=n_\orbit(\epsilon_n),\qquad n\in\N_{>0},\] 
 	is obtained from $\tilde\Theta_\orbit(s)$ through the integral transform
  	\begin{equation}\label{eq:tildeThetaFatou}
 		T_\orbit(\epsilon)=\tfrac{1}{2\pi\i}\int_{\Circ(\R_{\leq 0})}\tfrac{\tilde\Theta_\orbit(s)}{s}\,\e^{s\,\tau(\epsilon)}\d s=:\widetilde\Laplace\big\{\tfrac{\tilde\Theta_\orbit(s)}{2\pi\i\, s}\big\}(\epsilon),
 	\end{equation}
 	where $\Circ(\R_{\leq 0})$ is the Hankel contour encircling the ray $\R_{\leq 0}$ $($Figure~\ref{figure:1}$)$.
\end{cx}
 
\begin{proof}
By Proposition~\ref{prop:tildeTheta} and Theorem~\ref{thm:1}, the sectorial Fatou coordinate $\psi_{\tilde\orbit}(\tilde x)=\psi_\orbit\circ\phi^{\circ(-1)}(\tilde x)$ for $\tilde f(\tilde x)=\phi\circ f\circ\phi^{\circ(-1)}(\tilde x)$ is given by the integral
\begin{equation*}
 		\psi_{\tilde\orbit}(\tilde x)=\tfrac{1}{2\pi\i}\int_{\Circ(\R_{\leq 0})}\tfrac{\tilde\Theta_\orbit(s)}{s}\,\e^{s\,t(\tilde x)}\,\d s.
 	\end{equation*}
By \eqref{eq:phi} and by definition $T_\orbit(\epsilon)=\psi_\orbit\circ g^{\circ(-1)}(2\epsilon)=\psi_{\tilde\orbit}\circ\tilde g^{\circ(-1)}(2\epsilon)$, while $\tau(\epsilon)=t\circ g^{\circ(-1)}(2\epsilon)$, 	
from which \eqref{eq:tildeThetaFatou} follows.
 \end{proof}

Since $f$ and $\tilde f$ are analytically conjugated, the two functions $\Theta_\orbit$ and $\Theta_{\tilde\orbit}$ carry the same information about the analytic class of $f$. Hence  Theorem~\ref{cor:1} becomes:

\begin{cx}\label{cor:2}
	The Birkhoff--\'Ecalle--Voronin invariants of the germ $f$ can be read from the leading term coefficients at the singularities of the analytic continuation of  $\tilde\Theta_\orbit(s)$.
\end{cx}

\section{Hyperfunctions, Laplace transforms and Borel summability}

The theory of hyperfunctions in dimension one by  K\"othe \cite{Kothe} and Sato \cite{Sato,Sato1}, and its further development by Komatsu \cite{Komatsu},
provides a good framework for treatment of the distributions that arise as Laplace transforms of analytic functions, and for \'Ecalle's theory of resurgence.
In particular, it allows, to avoid dealing with Dirac distributions through the use of residue theorem.
Let us here recall some basics and fix the notation.

\begin{notation}
	For $c\in\R$ we denote $c\pls=c+\epsilon$ (resp. $c\mns=c-\epsilon$) with an arbitrarily small (infinitesimal) $\epsilon>0$.
\end{notation}

\subsection{Hyperfunctions.}
Let $\Omega\subseteq \e^{\i\alpha}\R$ be a locally closed subset of a line $\e^{\i\alpha}\R$ in the complex plane. Namely, for us $\Omega$ will be one of the following: the whole line $\e^{\i\alpha}\R$, a half-line $\e^{\i\alpha}\R_{\geq 0}$, or a single point $\{\omega\}$.  
We denote
\begin{itemize}[wide=0pt, leftmargin=\parindent]
	\item $\Cal O(\Omega)=\varinjlim_{U\supset\Omega\ \text{open}}\Cal O(U)$ the ring of functions analytic on neighborhoods of $\Omega\subset\C$,
	\item $\Cal O(\C\sminus\Omega,\Omega)=\varinjlim_{U\supset\Omega\ \text{open}}\Cal O(U\sminus\Omega)$ the ring of functions 
	 analytic on deleted neighborhoods of $\Omega$,
	\item $\Cal B(\Omega)=\Cal O(\C\sminus\Omega,\Omega)\big/\Cal O(\Omega)$ the ring of \emph{hyperfunctions with support on $\Omega$}, consisting 
	 of equivalence classes of functions $[\tilde H]_\Omega=\tilde H + \Cal O(\Omega)$.
\end{itemize}
For each hyperfunction $[\tilde H]_\Omega\in \Cal B(\Omega)$ we define its \emph{boundary value} $\BV[\tilde H(s)]_\Omega$, $s\in\Omega$, as the difference 
``$\tilde H(s-\i \e^{\i\alpha}(0\pls))-\tilde H(s+\i \e^{\i\alpha}(0\pls) )$'' understood in a ``distributional sense'' as a shorthand for:
\begin{align*}\label{eq:dai}
\int_\Omega\!\! G(s)\BV[\tilde H(s)]_\Omega\! \d s&=
\int_\Omega\!\! G(s)\tilde H(s\!-\!\i \e^{\i\alpha}(0\pls))\d s- \!\int_\Omega\!\! G(s)\tilde H(s\!+\!\i \e^{\i\alpha}(0\pls))\d s\\
&=\int_{\Circ(\Omega)}\!\!\!\! G(s)\tilde H(s)\d s,\qquad 	\text{for any } \ G\in \Cal O(\Omega),\nonumber
\end{align*}
where $\Circ(\Omega)$ is a contour encircling $\Omega$ in the positive direction. Namely, depending on the set $\Omega$:
\begin{itemize}[wide=0pt, leftmargin=\parindent]
	\item $\Circ(\e^{\i\alpha}\R)=\e^{\i\alpha}(-\i (0\pls)+\R)\cup \e^{\i\alpha}(\i(0\pls)-\R)$ is the difference between a pair of oriented parallel lines,  below minus above of $\e^{\i\alpha}\R$, 
	\item $\Circ(\e^{\i\alpha}\R_{\geq 0})$ is the Hankel contour around the ray $\e^{\i\alpha}\R_{\geq 0}$,
	\item $\Circ(\{\omega\})$ is a small circle around the point $\omega$.
\end{itemize}
Thus we distinguish between an analytic function $\tilde H(s)$ on a complement of $\Omega$, the associated hyperfunction $[\tilde H(s)]_\Omega$, and the distribution $\BV[\tilde H(s)]_\Omega$ on $\Omega$ defined by $[\tilde H(s)]_\Omega$.

\subsection*{Examples.}
\begin{enumerate}[wide=0pt, leftmargin=\parindent]
	\item The \emph{Heaviside function} $\mathbbm{1}_{\e^{\i\alpha}\R_{\geq 0}}(s)	=\begin{cases} 1, & s\in \e^{\i\alpha}\R_{\geq 0},\\ 0,& s\in \e^{\i\alpha}\R_{<0},\end{cases}\ $ 
	can be identified (almost everywhere) with the distribution 
	$\BV[\tfrac{1}{2\pi\i }\log s]_{\e^{\i\alpha}\R_{\geq 0}}$.
	\item The \emph{Dirac distribution} $\delta_\omega(s)$ at a point $\omega\in\Omega$ and its \emph{derivatives} $\delta_\omega^{(m)}(s)$, $m\in\Z_{>0}$:
	\[\delta_\omega(s)=\BV\big[\tfrac{1}{2\pi\i \,(s-\omega)}\big]_\Omega,\qquad 
	\delta_\omega^{(m)}(s)=\big(\tdd{s}\big)^m\delta_\omega(s)=\BV\big[\tfrac{(-1)^{m}m!}{2\pi\i \,(s-\omega)^{m+1}}\big]_\Omega.\]
	\item If $\tilde H(s)=\sum_{n\in\Z}a_n(s\!-\!\omega)^n$ is a convergent Laurent series at $\omega\in\C$, then 
	\[\BV[\tilde H(s)]_{\{\omega\}}=\sum_{n=-1}^{-\infty}a_{n}\BV[(s\!-\!\omega)^{n}]_{\{\omega\}}=2\pi\i \sum_{m=0}^{+\infty}a_{-1-m}\tfrac{(-1)^m}{m!}\delta_\omega^{(m)}(s).\]
	\item Let $\tilde N(s)=\tfrac{1}{2\pi\i}\log(\e^{2\pi\i\,s}-1)$ on $\C\sminus\big(\R_{\leq 0}\cup\R_{\geq 1}\big)$. Then the distribution $\BV[\tilde N(s)]_\R$ can be identified with the \emph{floor function} $\lfloor s\rfloor$ almost everywhere on $\R$.
\end{enumerate}

In the setting of this paper, $\tilde H$ will have at most a discrete set of singularities on $\Omega$, denoted $\Sing[\tilde H]_\Omega$, 
and the distribution $\BV[\tilde H]_\Omega$ will take the form of a sum of a piece-wise analytic function
on $\Omega\sminus\Sing[\tilde H]_\Omega$, 
called the \emph{variation} of $\tilde H$, given by the difference of the values of $\tilde H$ on the two sides of $\Omega$,
and of a sum of higher order Dirac distributions at points of $\Sing[\tilde H]_\Omega$, coming from the ``polar terms'' in $\tilde H$.



\medskip
Similarly, one can also consider \emph{germs of hyperfunctions} at a point $\omega\in\Omega$,
\[\Cal B(\Omega,\,\omega)=\Cal O(\C\sminus\Omega,\,\omega)/\Cal O(\C,\,\omega),\]
where $\Cal O(\C\sminus\Omega,\,\omega)=\varinjlim_{U\ni\omega}\Cal O(U\sminus\Omega)$ consist of germs of functions analytic outside of $\Omega$ on neighborhoods of $\omega$, and
$\Cal O(\C,\,\omega)$ are germs of functions analytic on neighborhoods of $\omega$.
For $\tilde H(s)\in\Cal O(\C\sminus\Omega,\,\omega)$, one then denotes $[\tilde H(s)]_{(\Omega,\, \omega)}\in \Cal B(\Omega,\omega)$ the germ of hyperfunction, and $\BV[\tilde H(s)]_{(\Omega,\, \omega)}$ the associated germ of distribution.\label{p:germofhyperfunction}

\subsection{Borel and Laplace transforms} 

The \emph{two-sided Laplace transform} of a function $H(s)$ on a line $\e^{\i\alpha}\R$,  or of a distribution $H(s)=\BV[\tilde H(s)]_{\e^{\i\alpha}\R}$,  is given by the following integral:%
\footnote{We use the kernel $\e^{st}$ in the Laplace transforms and $\e^{-st}$ in the Borel transforms, which is opposite to the usual convention. The reason is to have in the Borel transforms the same kernel $\e^{-st}$ as in the definition of the dynamic theta function $\Theta_\orbit(s)$.}\label{p:borel}
\begin{equation}\label{eq:Laplace2sided}
	h(t)=\widetilde\Laplace_\alpha\{H\}(t):=\int_{\e^{\i\alpha}\R} H(s)\,\e^{st}\d s=\int_{\Circ(\e^{\i\alpha}\R)}\tilde H(s)\,\e^{st}\d s.
\end{equation} 
If $H(s)$ is a function and if there exist $B<A\in\R$ such that $\sup_{\xi\in\R}|H(\e^{\i\alpha}\xi)(\e^{A\xi}+\e^{B\xi})|<+\infty$, then \eqref{eq:Laplace2sided}
converges absolutely and is analytic on the strip $\{B<\re(\e^{\i\alpha}t)<A\}\subseteq\C$.

The \emph{one-sided Laplace transform} of a function $H(s)$ on a ray $\e^{\i\alpha}\R_{\geq 0}$,  or of a distribution $H(s)=\BV[\tilde H(s)]_{\e^{\i\alpha}\R_{\geq 0}}$, is 
\begin{equation}\label{eq:Laplace1sided}
	h(t)=\Laplace_\alpha\{H\}(t):=\int_{\e^{\i\alpha}\R_{\geq 0}} H(s)\,\e^{st}\d s=\int_{\Circ(\e^{\i\alpha}\R_{\geq 0})}\tilde H(s)\,\e^{st}\d s.
\end{equation} 
If $H(s)$ is a function and $\sup_{\xi\in\R_{\geq 0}}|H(\e^{\i\alpha}\xi)\e^{A\xi}|<+\infty$ for some $A\in\R$ then \eqref{eq:Laplace1sided} converges absolutely and is analytic on the half-plane $\{\re(\e^{\i\alpha}t)<A\}\subseteq\C$.

In both cases \eqref{eq:Laplace2sided}, resp. \eqref{eq:Laplace1sided}, the inverse transform is the \emph{Borel transform}
\begin{equation}\label{eq:Borel2sided}
	H(s)=\Borel_\alpha\{h\}(s)=\tfrac{1}{2\pi\i }\VP\int_{C+\i \e^{-\i \alpha}\R} h(t)\,\e^{-st}\d t,\qquad s\in \e^{\i\alpha}\R,
\end{equation} 
where $\VP\int_{C+\i \e^{-\i \alpha}\R}=\lim_{N\to+\infty}\int_{C-\i \e^{-\i \alpha}N}^{C+\i \e^{\i\alpha}N}$ stands for the ``Cauchy principal value'', and where $C\in\C$ is such that the line $C+\i \e^{-\i \alpha}\R$ is contained in the strip $\{B<\re(\e^{\i\alpha}t)<A\}$, resp. in the half-plane $\{\re(\e^{\i\alpha}t)<A\}$.
If the situation allows, the integration contour $ C+\i \e^{-\i \alpha}\R$  to make the integral absolutely convergent.

In the one-sided case \eqref{eq:Laplace1sided}, if $h(t)$ is analytic on a half-plane  $\{\re(\e^{\i\alpha}t)<A\}$ and has a sub-exponential growth along any ray 
$C+\e^{-\i \beta}\R_{\geq 0}$ pointing strictly inside the half-plane,
then $H(s)=\BV[\tilde H(s)]_{\e^{\i\alpha}\R_{\geq 0}}$ for
\begin{equation*}\label{eq:Borel1sided}
\tilde H(s)=\widetilde{\Borel}_\alpha\{h\}(s)=\tfrac{1}{2\pi\i}\int_{C+\e^{-\i \beta}\R_{\geq 0}}\!\!\! h(t)\,\e^{-st}\d t,\quad
\text{for any }\ \beta\in\
 ]\alpha\!+\!\tfrac{\pi}{2},\alpha\!+\!\tfrac{3\pi}{2}[.
\end{equation*}
Indeed, varying $\beta$ in the interval extends the domain of definition of $\tilde H(s)$ from the half-plane $\{\re(\e^{-\i\beta}s)<0\}$ to $\C\sminus \e^{\i\alpha}\R_{\geq 0}$, and
\begin{align*}
	H(s)&=\tfrac{1}{2\pi\i }\!\int_{C+\i \e^{-\i \alpha}\R_{\geq 0}}\!\!\!\!\! h(t)\,\e^{-st}\d t - \tfrac{1}{2\pi\i }\!\int_{C-\i \e^{-\i \alpha}\R_{\geq 0}}\!\!\!\!\! h(t)\,\e^{-st}\d t =\BV[\tilde H(s)]_{\e^{\i\alpha}\R_{\geq 0}}.
\end{align*}
See \cite{Komatsu} for more details.

\subsection{Borel summation.} 
Let $t(x)=\psimod(x)=-\tfrac{1}{ak}x^{-k}$ be as in \eqref{eq:psimod}.

For a formal series $\hat h(x)=\sum_n h_nx^n$, its
\emph{minor Borel transform  in a direction $\alpha\in\R$} with weight $t(x)$ is defined term-wise as
\begin{equation*}\label{eq:Boreltransform}
	\Borel_\alpha\{\hat h\}(s)=\sum_n h_n\Borel_\alpha\{x^n\}(s),\qquad s\in(\e^{\i\alpha}\R_{\geq 0},0),
\end{equation*}	
where $\Borel_\alpha$ is the Borel transform \eqref{eq:Borel2sided} with $t=t(x)$.
In particular, 
\begin{equation}\label{eq:Borelxn}
	\Borel_\alpha\{x^\nu\}(s)=\begin{cases}  \frac{1}{ak\,\Gamma(\frac{\nu}{k})}\Big(\frac{s}{ak}\Big)^{\frac{\nu}{k}-1},& \tfrac{\nu}{k}\in\C\sminus \Z_{\leq 0},\\[6pt]
	(ak)^{-\frac{\nu}{k}}\delta_0^{(-\frac{\nu}{k})}(s)=\big(ak\tdd{s}\big)^{-\frac{\nu}{k}}\delta_0(s),& \tfrac{\nu}{k}\in \Z_{\leq 0},
\end{cases}
\end{equation}
for $s\in \e^{\i\alpha}\R_{\geq0}$.

\begin{remark}\label{rem:Borel}
The minor Borel transform $\Borel_\alpha\{x^\nu\}(s)$ \eqref{eq:Borelxn} for $\tfrac{\nu}{k}\notin\Z_{\leq 0}$ is standardly considered on the ray $s\in\e^{\i\alpha}\R_{\geq 0}$.
For $s\in\e^{\i\alpha}\R_{\geq 0}$ the integrating contour $C+\i\e^{-\i\alpha}\R$ with $\Re(\e^{\i\alpha}C)<0$ in \eqref{eq:Borel2sided} in the $t(x)$-coordinate may be changed to the negatively oriented Hankel contour around $\e^{-\i\alpha}\R_{\geq 0}$ to make the integral absolutely convergent.
However, if considered on the whole line $\e^{\i\alpha}\R$, then for $s\in\e^{\i\alpha}\R_{<0}$ one has $\Borel_\alpha\{x^\nu\}(s)=0$.
Indeed, the integrating contour in the case $s\in\e^{\i\alpha}\R_{< 0}$ may be changed to the positively oriented Hankel contour around $C+\e^{-\i\alpha}\R_{\leq 0}$, for which the integral, by the residue theorem, is equal to $0$. 
\end{remark}

Similarly, the \emph{major Borel transform} of a formal series $\hat h$ is defined term-wise as:
\[\widetilde\Borel_\alpha\{\hat h\}(s)=\sum_n h_n\widetilde\Borel_\alpha\{x^n\}(s),\] 
where we define
\begin{equation}\label{eq:tildeBorelxn}
		\widetilde\Borel_\alpha\{x^\nu\}(s):=\begin{cases}
			\tfrac{1}{2\pi\i\,ak\,\Gamma(\frac{\nu}{k})} \Big(\tfrac{s}{ak}\Big)^{\frac{\nu}{k}-1}\log s,& \tfrac{\nu}{k}\in\Z_{>0}, \\[6pt]	
\tfrac{-\Gamma(1-\frac{\nu}{k})}{2\pi\i\,ak} \Big(\tfrac{-s}{ak}\Big)^{\frac{\nu}{k}-1},& \tfrac{\nu}{k}\in\C\sminus\Z_{>0}. 
		\end{cases}
\end{equation}

Then, by \eqref{eq:Borelxn} and \eqref{eq:tildeBorelxn},  $\Borel_\alpha\{x^\nu\}(s)=\BV\big[\widetilde\Borel_\alpha\{x^\nu\}(s)\big]_{(\e^{\i\alpha}\R_{\geq0},\,0)}$ for all $\nu\in\C$.
Hence 
$\Borel_\alpha\{\hat h\}(s)=\BV\big[\widetilde\Borel_\alpha\{\hat h\}(s)\big]_{(\e^{\i\alpha}\R_{\geq0},\,0)}$ if either of the two series converges.

\smallskip
The \emph{Borel sum in a direction $\alpha\in\R$} with weight $t(x)$ of a formal series $\hat h(x)$ is defined by:
\begin{align*}
	\BSum_\alpha\{\hat h\}(x)&:= \Laplace_\alpha\{H(s)\}(x)=\int_{\e^{\i\alpha} \R_{\geq 0}}H(s)\,\e^{s\,t(x)}\d s ,& H(s)=\Borel_\alpha\{\hat h(x)\}(s),\\
	&:=\widetilde\Laplace_\alpha\{\tilde H(s)\}(x)=\int_{\Circ(\e^{\i\alpha}\R_{\geq0})\vphantom{\big |}}\!\!\!\!\!\!\!\! \tilde H(s)\,\e^{s\,t(x)}\d s ,& \tilde H(s)=\widetilde\Borel_\alpha\{\hat h(x)\}(s),
\end{align*}
assuming that $H(s)$ is defined on an initial segment of $\e^{\i\alpha}\R_{\geq 0}$, and is analytically extendable to the whole ray $\e^{\i\alpha}\R_{\geq 0}$
with at most exponential growth there, i.e. that $|H(s)|\leq C \e^{-A_\alpha|s|}$ for some $A_\alpha\in\R$, in which case $\BSum_\alpha\{\hat h\}(x)$ is defined on  $\{\re(\e^{\i\alpha}t(x))<A_\alpha\}$.

\begin{remark}\label{rem:log}
Let us calculate $\Borel_\alpha\{\log t(x)\}(s)$.
\[\widetilde\Borel_\alpha\{\log t(x)\}(s)=\tfrac{1}{2\pi\i}\int_{1-\e^{-\i\alpha}\R_{\geq 0}}\!\!\!\!\log t\,\e^{-st}\d t=\tfrac{1}{2\pi\i\,s}\int_1^{-\e^{\i\alpha}\infty}\!\!\! t^{-1}\e^{-st}\d t=\tfrac{E_1(s)}{2\pi\i\, s},\]
where $E_1(s)=\int_1^{+\infty}t^{-1}\e^{-st}\d t=-\gamma-\log s-\sum_{l=1}^{+\infty}\frac{(-s)^l}{l\cdot l!}$ is the exponential integral and $\gamma=0.5772...$ the Euler constant.
Since $\tfrac{-1}{2\pi\i\,s}\sum_{l=1}^{+\infty}\frac{(-s)^l}{l\cdot l!}$ is an entire function, 
\[\Borel_\alpha\{\log t(x)\}(s)=\BV\big[\widetilde\Borel_\alpha\{\log t(x)\}(s)\big]_{\e^{\i\alpha}\R_{\geq 0}}=-\gamma\,\delta_0(s)-\tfrac{1}{s}.\]
\end{remark}

\section{Fatou coordinates and analytic invariants of parabolic germs}

Here we give a quick overview of some of the basics of the Birkhoff--\'Ecalle--Voronin theory of analytic classification of germs of diffeomorphisms of $(\C,0)$ tangent to the identity,
discovered independently by Birkhoff \cite{Birkhoff}, \'Ecalle \cite{EcalleBouillot,Ecalle1,Ecalle,Ecalle2} and Voronin \cite{Voronin}.
There are many excellent references for this theory, such as \cite{ Bracci, IlYa, Loray}.

\subsection{Formal and sectorial Fatou coordinates.}\label{sec:3.1}

Let $f$ be a germ \eqref{eq:f} and $\fmod=\exp(\ximod)$  be its model \eqref{eq:psimod}. 
The germ $f$, being tangent to the identity, possesses a unique formal infinitesimal generator: a formal vector field $\hat \xi(x)=\big(ax^{k+1}+\hot(x)\big)\tdd{x}$,  such that the Taylor series $\hat f(x)$ for $f(x)$ is equal to the formal time-$1$ flow of $\hat\xi$: 
\[\hat f(x)=\exp(\hat\xi)(x)=\sum_{n=0}^{+\infty}\tfrac{1}{n!}\hat\xi^{n}.x \,,\]
(here $\hat\xi.\hat h$ denotes the formal Lie derivative associated to $\hat\xi$ applied on some formal power series $\hat h$).

The formal antiderivative $\hat\psi(x)=\int\hat \xi^{-1}$ of the dual formal meromorphic 1-form $\hat\xi^{-1}$
defines a \emph{formal Fatou coordinate} for $f$ of the form 
\begin{equation}\label{eq:formalFatou}
	\hat\psi(x)=r_{-k}x^{-k}+\ldots+r_{-1}x^{-1}+\rho\log x+C+\sum_{j=1}^{+\infty}r_jx^j,\qquad C\in\mathbb C,
\end{equation}
where $\rho$ is the \emph{residual invariant} of $f$ given by	
\begin{equation}\label{eq:resit}
		\rho=\res_{x=0}\hat\xi^{-1}=\res_{x=0}\left(\mfrac{\d x}{f(x)-x}\right)+\mfrac{k+1}{2},
\end{equation}
and the coefficients $r_m$ are uniquely determined by $f$.
It satisfies the \emph{Abel equation}
\begin{equation*}\label{eq:Fatourelation}
	\hat\psi\circ f=\hat\psi+1,
\end{equation*} 
where $f$ is identified with its Taylor series at $0$.

A \emph{sectorial Fatou coordinate} for $f$ is an analytic function $\psi_W$ on some sector $W$ that satisfies the \emph{Abel equation} $\psi_W\circ f=\psi_W+1$ and
such that $\psi_W(x)-\big(r_{-k}x^{-k}+\ldots+r_{-1}x^{-1}+\rho\log x\big)$ is bounded on $W$ (equivalently, one could ask $\psi_W:W\to \mathbb C$ to be injective).

One can always prenormalize $f$ through an analytic change of coordinate so that it takes the form $f(x)=x+ax^{k+1}+a^2(\tfrac{k+1}{2}-\rho)x^{2k+1}+\hot(x)$,
in which case the formal Fatou coordinate takes the form
\begin{equation}\label{eq:formalFatou1}
 \hat\psi(x)=t(x)-\tfrac{\rho}{k}\log t(x)+C'+\hat r(x), \quad
 \text{with }\ \hat r(x)=\sum_{n=1}^{+\infty}r_nx^n\quad \text{and }\ C'\in\mathbb C. 
\end{equation}
Assuming $f$ to be in a prenormalized form is not essential, but it is convenient since it simplifies the discussion.
The formal series $\hat r(x)$ is in general divergent, but always Borel summable of order $k$.

\begin{theorem}[Leau\,\cite{Leau}, Fatou\,\cite{Fatou}, Birkhoff\,\cite{Birkhoff}, Kimura\,\cite{Kimura}, \'Ecalle\,\cite{Ecalle}, Voronin\,\cite{Voronin}] \label{thm:sectorialFatou}
The germ $f(x)$ \eqref{eq:f} possesses a  \emph{``cochain'' of sectorial Fatou coordinates} $\psi_{W_j}(x)$ given by the Borel sums
\begin{equation}\label{eq:psialpha}
	\psi_{W_j}(x):=\BSum_{\alpha_j}\{\hat\psi\}(x)=t(x)-\tfrac{\rho}{k}\log t(x)+C'+\BSum_{\alpha_j}\{\hat r\}(x),  
\end{equation}
$\alpha_j\in\  ](j-\tfrac12)\pi,(j+\tfrac12)\pi[$, on a covering by $2k$ sectors, called \emph{Leau--Fatou petals},
\[W_{j}\subset\big\{x:\ \arg(t(x))\in\ ](j-1)\pi,(j+1)\pi[\,\big\},\qquad j=1,\ldots,2k,\]
see Figure~\ref{figure:parabolic}. 
\end{theorem}

The map $x\mapsto t(x)$ is a $k$-sheeted branched covering of a neighborhood of $t=\infty$.
Therefore the formal series $\hat r(x)$ in \eqref{eq:formalFatou1} becomes $k$-times ramified as a series in $t$.
The change of sheet $t\mapsto \e^{-2\pi\i}t$ on a neighborhood of $t=\infty$ corresponds to the opposite change of sheet $s\mapsto \e^{2\pi\i}s$ on a neighborhood of $s=0$ in the Borel plane:
\[\Borel_\alpha\{\hat r(x)\}(s)=\Borel_{\alpha+2\pi}\{\hat r(\e^{\frac{2\pi\i}{k}}x)\}(\e^{2\pi\i}s).\]
In particular, this implies that
\[\BSum_{\alpha_j+2k\pi}\{\hat r\}(\e^{2\pi\i }x)=\BSum_{\alpha_j}\{\hat r\}(x),\quad \BSum_{\alpha_j+2k\pi}\{\hat\psi\}(\e^{2\pi\i }x)=\BSum_{\alpha_j}\{\hat \psi\}(x)+2\pi\i \rho.\]
This explains why $2k$ Borel sums \eqref{eq:psialpha} are needed.

\medskip

Theorem~\ref{thm:sectorialFatou} is a consequence of:

\begin{theorem}[\'Ecalle \cite{Ecalle2, Ecalle3}]\label{thm:Ecalle}
The minor Borel transform $\Borel_\alpha\{\hat r\}(s)$ of the formal series $\hat r(x)$ in \eqref{eq:formalFatou1}
extends analytically from the ray $\e^{\i\alpha}\R_{\geq 0}$ to the universal cover $\widetilde{\C\sminus 2\pi\i\Z}$, and has at most exponential growth when $s\to \infty$ along any non-vertical direction on any sheet.

Furthermore, if  $\rho=0$, then $\Borel_\alpha\{\hat r\}(s)$ is \emph{simple resurgent}, meaning that all its singularities $\omega\in2\pi\i\Z$, when accessed from any sheet, have the form 
\[(s\!-\!\omega)^{-1}\C\{(s\!-\!\omega)^{\frac{1}{k}}\}
+\log(s\!-\!\omega)\,\C\{s\!-\!\omega\}.\]
\end{theorem}
See \cite[Section 9]{Ecalle2}, \cite{Malg-resurgence}, and the papers \cite{DudkoSauzin1, DudkoSauzin2} of Dudko \& Sauzin with a new short proof.

\begin{remark}\label{rem:continuation}
The minor Borel transform of the formal Fatou coordinate \eqref{eq:formalFatou1}
\begin{equation}\label{eq:BorelFatou}
\Psi(s)=\Borel_\alpha\{\hat\psi\}(s)=\delta_0^{(1)}(s)+\tfrac{\rho}{k}s^{-1}+ (C'+\tfrac{\rho\gamma}{k})\,\delta_0(s)+\Borel_\alpha\{\hat r\}(s).
\end{equation}
is a distribution, defined by $\BV[\tilde\Psi(s)]_{\e^{\i\alpha}\R_{\geq 0}}$, where
\[\tilde\Psi(s):=\widetilde\Borel_\alpha\{\hat\psi\}(s)=
-\tfrac{1}{2\pi\i \,s^2}-\tfrac{\rho}{2\pi\i k\, s}E_1(s)+ \tfrac{C'}{2\pi\i \,s} +\widetilde\Borel_\alpha\{\hat r\}(s),\]
see Remark~\ref{rem:log}.
Since $\Psi(s)$ is a sum of a Dirac part $\delta_0^{(1)}(s)+(C'+\tfrac{\rho\gamma}{k})\,\delta_0(s)$ with an analytic function $\tfrac{\rho}{k}s^{-1}+ \Borel_\alpha\{\hat r\}(s)$, one can as well talk about the analytic continuation of $\Psi(s)$. The Dirac part stays unchanged under continuation: indeed $\delta_0^{(1)}(s)+(C'+\tfrac{\rho\gamma}{k})\,\delta_0(s)=\BV\big[\tfrac{-1}{2\pi\i\,s^2}+\tfrac{C'+\frac{\rho\gamma}{k}}{2\pi\i\,s}]_{\e^{\i\alpha}\R_{\geq 0}}$, and the function $\tfrac{-1}{2\pi\i,s^2}+\tfrac{C'+\frac{\rho\gamma}{k}}{2\pi\i\,s}$ does not change under continuation, while any deformation of the cut $\e^{\i\alpha}\R_{\geq 0}$ is end-point homotopic to itself. 
On the other hand, let us note that the continuation of the term $\tfrac{\rho}{k}s^{-1}=\BV\big[\tfrac{\rho}{2\pi\i k\,s}\log s\big]_{\e^{\i\alpha}\R_{\geq 0}}$ along a simple loop around $0$ changes it by the addition of $2\pi\i \tfrac{\rho}{k}\delta_0(s)$.
\end{remark}

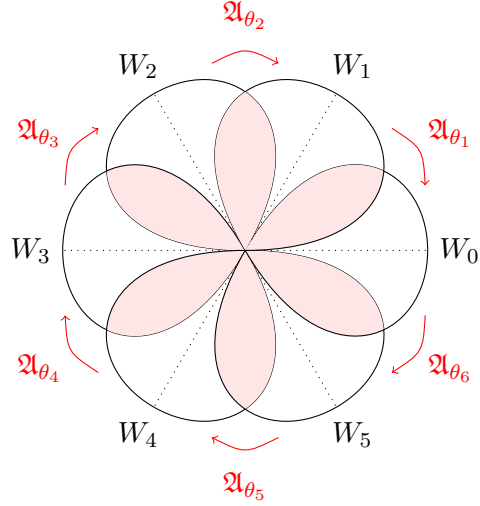
\begin{figure}[t]
\centering	
\begin{tikzpicture}[scale=0.6]
\foreach \n  in {1,...,6} {
\pgfmathtruncatemacro\m{\n-1};
\begin{scope}[rotate=\n*60-30, even odd rule] 
\draw[dotted] (0,0) -- (0,4); \draw (-30:4.7) node{$W_{\m}$};	\draw[->, red] (10:4.2) .. controls (0:4.5) .. (-10:4.2); 
\draw[red] (0:5.2) node{$\Aa_{\theta_{\n}}$};	
\draw[clip] (0,0) .. controls (3,{sqrt(3)}) and (2,4) .. (0,4) .. controls  (-2,4) and (-3,{sqrt(3)}) .. (0,0); 
\fill[red!10, rotate=60]  (0,0) .. controls (3,{sqrt(3)}) and (2,4) .. (0,4) .. controls  (-2,4) and (-3,{sqrt(3)}) .. (0,0); 
\end{scope}
}
\end{tikzpicture}
\caption{The petals $W_j$, $j\in\Z_{2k}$ (here, $k=3$ and $a\in\R$), and the transition maps $\Aa_{\theta_j}$ \eqref{eq:Apms}  on the intersections.}
\label{figure:parabolic}
\end{figure}

\subsection{Transition maps and moduli of analytic classification.}\label{sec:3.2}

For each singular direction $\theta_j=-\frac{\pi}{2}+j\pi$, $j=1,\ldots,2k$, one has a \emph{transition map} (\emph{connector}) $\Aa_{\theta_j}$ between the two Fatou coordinates in the adjacent directions
\begin{equation*}\label{eq:Apmsdefinition}
	\BSum_{\theta_j\mns}\{\hat\psi\}=\Aa_{\theta_j}\circ\BSum_{\theta_j\pls}\{\hat\psi\}.
\end{equation*}
Since each $\Aa_{\theta_j}(t)$ commutes with the translation $t\mapsto t+1$ and is asymptotic to the identity
(since all the sectorial Fatou coordinates are asymptotic to the same formal Fatou coordinate), 
it can be expressed in terms of a Fourier series of the form
\begin{equation}\label{eq:Apms}
	\Aa_{\theta_j}(t)=t+\sum_{\omega\in \Sing_\theta}A_{\omega,j}\, \e^{\omega t},\qquad A_{\omega,j}\in\C,
\end{equation}
where 
\[\Sing_{\theta_j}= \e^{\i\theta_j}\R_{>0}\cap2\pi\i \Z.\]
In another words
\begin{equation}\label{eq:A+-}
	 \BSum_{\theta_j\mns}\{\hat\psi\}-\BSum_{\theta_j\pls}\{\hat\psi\}=\sum_{\omega\in \Sing_\theta}A_{\omega,j}\, \e^{\omega \BSum_{\theta\pls}\{\hat\psi\}}.
\end{equation}

The collection of the transition maps 
\[\big\{\Aa_{\theta_j}(t)\big\}_{j=1}^{2k},\qquad \theta_j=(j\!-\!\tfrac12)\pi,\] 
is uniquelly determined up to conjugation by translations $t\mapsto t+C$, $C\in\C$, 
\begin{equation}\label{eq:cocycleequivalence}
	\Aa_{\theta_j}\sim {\Aa}'_{\theta_j}\ \Longleftrightarrow\ \exists\, C\in\C: \ {A}'_{\omega,j}=A_{\omega,j}\, \e^{\omega C}\quad\text{for all }\ \omega\in\Sing_{\theta_j}.
\end{equation} 
corresponding to the liberty of the choice of the constant in the formal Fatou coordinate, $\hat\psi'=\hat\psi-C$ and  ${A}'_{\omega,j}\, \e^{\omega \hat\psi'}=A_{\omega,j}\, \e^{\omega \hat\psi}$, cf. \eqref{eq:A+-}.

The resulting equivalence class 
$\big\{ \Aa_{\theta_j}\big\}_{j=1}^{2k}\big/\C$
is called a \emph{cocycle} or \emph{Birkhoff--\'Ecalle--Voronin modulus}.
It is an analytic invariant of $f$  which expresses the obstruction to convergence of the formal Fatou coordinate. 
It was initially described by G.D.~Birkhoff \cite{Birkhoff} and later independently by J.~\'Ecalle \cite{Ecalle1,Ecalle} and S.M.~Voronin \cite{Voronin}.

~
\goodbreak

\begin{theorem}[Analytic classification]\label{thm:Ecalle1}~
	\begin{enumerate}[wide=0pt, leftmargin=\parindent]
		\item \textnormal{(Birkhoff \cite{Birkhoff}, \'Ecalle \cite{Ecalle1, Ecalle}, Voronin \cite{Voronin}).}
		Two germs $f$, $f'$ that are formally tangent-to-identity equivalent $($i.e. have the same model $\fmod$ and the same iterative residue $\rho)$ are analytically tangent-to-identity equivalent if and only if their cocycles $\big\{ \Aa_{\theta_j}\big\}_{j=1}^{2k}\big/\C$, $\big\{{\Aa}'_{\theta_j}\big\}_{j=1}^{2k}\big/\C$ agree.
		\smallskip
		
		\item \textnormal{(\'Ecalle \cite{Ecalle2}, Malgrange \cite{Malg-diffeo}, Voronin \cite{Voronin}).}
		For each model $\fmod$, iterative residue $\rho$, and each collection of transition maps $\big\{\Aa_{\theta_j}\big\}_{j=1}^{2k}$ of the form \eqref{eq:Apms}, there exists an analytic map $f$ whose cocycle is represented by $\big\{\Aa_{\theta_j}\big\}_{j=1}^{2k}$.
	\end{enumerate}	
\end{theorem}

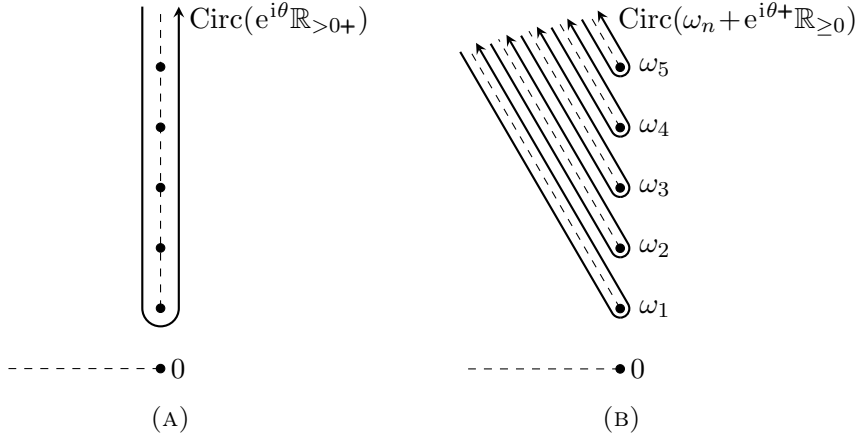
\begin{figure}[t]
	\begin{subfigure}{.3\textwidth} 
		\begin{tikzpicture}[scale=.8]
			\draw[dashed] (0,0) -- (-2.5,0);
			\draw[dashed] (0,1) -- (0,6);
			\draw[-stealth, thick] (-0.3,6) -- (-0.3,1)  arc[start angle=180, end angle=360,radius=0.3] -- (0.3,6);
			\foreach \n in {0,...,5} {\filldraw (0,\n) circle (2pt);};
			\draw (0,0) node[right]{$0$};
			\draw (0.3,5.8) node[right]{${\Circ(\e^{\i\theta}\R_{>0\pls})}$};
		\end{tikzpicture} 
		\caption{} 
	\end{subfigure}	
	\qquad\qquad
	\begin{subfigure}{.3\textwidth} 
		\begin{tikzpicture}[scale=.8]
			\draw[dashed] (0,0) -- (-2.5,0);
			\foreach \n in {1,...,5} {
				\begin{scope}
					\draw[dashed] (0,\n) -- +(120:6-\n);
					\filldraw (0,\n) circle (2pt); \draw (0.15,\n) node[right]{$\omega_\n$};
					\draw[-stealth, thick] (0,\n-0.15) arc[start angle=-90, end angle=30,radius=0.15] -- +(120:6-\n);
					\draw[thick] (0,\n-0.15) arc[start angle=-90, end angle=-150,radius=0.15] -- +(120:6-\n);
				\end{scope}
			};
			\filldraw (0,0) circle (2pt) node[right]{$0$};
			\draw (-0.2,5.8) node[right]{${\Circ(\omega_n\!+\!\e^{\i\theta\pls}\R_{\geq 0})}$};
		\end{tikzpicture}
		\caption{}
	\end{subfigure}	
	\caption{The integrating contour $\Circ(\e^{\i\theta}\R_{>0\protect\pls})$ in a singular direction $\theta$ and its deformations as an infinite union of integrating contours in a direction $\theta\protect\pls$. 	The dashed lines represent ramification cuts.}
	\label{figure:contour1}
\end{figure}

\subsection{Reading the coefficients \texorpdfstring{$A_{\omega,j}$}{A} in the Borel plane.}\label{sec:3.3}   

Let $\hat\psi$ be the formal Fatou coordinate of a prenormalized $f$, as in \eqref{eq:formalFatou1} and $\Psi(s)=\Borel_\alpha\{\hat\psi\}(s)$  \eqref{eq:BorelFatou} its minor Borel transform defined as a distribution on an initial segment of $\e^{\i\alpha}\R_{>0}$ and extended analytically (see Remark~\ref{rem:continuation}).

We show here that the invariants ${A}_{\omega,j}$ (the Fourier coefficients of $\Aa_{\theta_j}$ \eqref{eq:A+-}) can be read from the coefficient of the leading singular term of $\Psi(s)$ at the singularity $\omega$, when approached from the negative, resp. positive, side as in Figure~\ref{figure:contour1}.

For a singular direction $\theta_j$, we have
\begin{equation}\label{eq:differenceFatou}
	\begin{aligned}
		\BSum_{\theta_j\mns}\{\hat\psi\}-\BSum_{\theta_j\pls}\{\hat\psi\}&=\int_{\Circ(\e^{\i\theta_j}\R_{>0\pls})}\!\!\!\! \Psi(s)\,\e^{s\,t(x)}\d s\\
		&=\sum_{\omega\in\Sing_{\theta_j}}\int_{\Circ(\omega+\e^{\i\theta_j\pls}\R_{\geq 0})}\!\!\!\! \Psi(s)\,\e^{s\,t(x)}\d s\\
		&=\sum_{\omega\in\Sing_{\theta_j}}\!\!\!\!\Laplace_{\theta_j\pls}\{\BV[\Psi(s)]_{\omega+\e^{\i\theta_j\pls}\R_{\geq 0}}\},
	\end{aligned}
\end{equation}
see Figure~\ref{figure:contour1}.
Comparing with the identity \eqref{eq:A+-} we obtain
\begin{equation}\label{eq:Aomega}
	A_{\omega,j} \e^{\omega \BSum_{\theta_j\pls}\{\hat\psi\}}=\Laplace_{\theta_j\pls}\{\BV[\Psi(s)]_{\omega+\e^{\i\theta_j\pls}\R_{\geq 0}}\},\qquad \omega\in\Sing_{\theta_j}.
\end{equation} 
Indeed, since the direction  $\theta_j\pls$ of the Hankel contours $\Circ(\omega +\e^{\i\theta_j\pls}\R_{\geq 0})$ can be varied a bit, each of the shifted Laplace integrals 
\[\e^{-\omega t(x)}\Laplace_{\theta_j\pls}\{\BV[\Psi(s)]_{\omega+\e^{\i\theta_j\pls}\R_{\geq 0}}\}=
\Laplace_{\theta_j\pls}\{\BV[\Psi(s+\omega)]_{\e^{\i\theta_j\pls}\R_{\geq 0}}\}\]
is bounded on a sector in $t(x)$-coordinate of opening $>\pi$.
Therefore the equality between \eqref{eq:A+-} and \eqref{eq:differenceFatou} holds term-wise. 

Now, since
\[A_{\omega,j}\, \e^{\omega \BSum_{\theta_j\pls}\{\hat\psi\}}
=\Laplace_{\theta_j\pls}\big\{\Borel_{\theta\pls}\{A_{\omega,j}\,  \e^{\omega\hat\psi}\}(s)\big\},\]
the identity \eqref{eq:Aomega} rewrites as
\begin{equation}\label{eq:da1}
\begin{aligned}
	\BV[\Psi(s)]_{\omega+\e^{\i\theta_j\pls}\R_{\geq 0}}&=\Borel_{\theta_j\pls}\{A_{\omega,j}\, \e^{\omega \hat\psi}\}(s)\\
	&=\BV\big[\widetilde\Borel_{\theta_j\pls}\{A_{\omega,j}\, \e^{\omega \hat\psi}\}(s)\big]_{\omega+\e^{\i\theta_j\pls}\R_{\geq 0}}.	
\end{aligned}
\end{equation}
We compute:
\begin{equation}\label{eq:da2}
\begin{aligned}
	A_{\omega,j}\, \e^{\omega\hat\psi(x)}&=A_{\omega,j}\, \e^{\omega C'}\e^{\omega t(x)} t(x)^{-\frac{\omega\rho}{k}}\e^{\omega \hat r(x)}\\
&= {A}'_{\omega,j}\, \e^{\omega t(x)} \big( t(x)^{-\frac{\omega\rho}{k}}+\hot(x)\big),
\end{aligned}
\end{equation}
where the coefficients ${A}'_{\omega,j}:=A_{\omega,j}\, \e^{\omega C'}$ represent the same classifying cocycle, determined modulo \eqref{eq:cocycleequivalence}.
The multiplication by $\e^{\omega t(x)}$ acts under the Borel transformation as the shift operator $s\mapsto s-\omega$,
which shifts the origin to $\omega$ and the integration path $\e^{\i\theta_j\pls}$ of the Laplace transform to $\omega+\e^{\i\theta_j\pls}$. 
By \eqref{eq:da1} and \eqref{eq:da2} we have locally near $s=\omega\in\Sing_{\theta_j}$:
\begin{align*}
	&[\Psi(s)]_{(\omega+\e^{\i\theta_j\pls}\R_{\geq 0},\,\omega)}
	=\Big[\widetilde\Borel_{\theta_j}\big\{{A}'_{\omega,j} \big(t(x)^{-\frac{\omega\rho}{k}}\!+\hot(x)\big)\big\}(s\!-\!\omega)\Big]_{(\omega+\e^{\i\theta_j\pls}\R_{\geq 0},\,\omega)}\\[6pt]
	&=
\begin{cases}
	\Big[ \mfrac{{A}'_{\omega,j}\, \Gamma(1-\tfrac{\omega\rho}{k})}{2\pi\i} (s\!-\!\omega)^{\frac{\omega\rho}{k}-1}
	+\hot(s\!-\!\omega)\Big]_{(\omega+\e^{\i\theta_j\pls}\R_{\geq 0},\,\omega)},& \tfrac{\omega\rho}{k}\notin\Z_{> 0}	\\
	\Big[  \mfrac{{A}'_{\omega,j}\, \e^{\pi\i \frac{\omega\rho}{k}}}{2\pi\i\,\Gamma(\tfrac{\omega\rho}{k})}  (s\!-\!\omega)^{\frac{\omega\rho}{k}-1}\!\log(s\!-\!\omega)+\hot(s\!-\!\omega) \Big]_{(\omega+\e^{\i\theta_j\pls}\R_{\geq 0},\, \omega)},\!\!\!\!\!
	& \tfrac{\omega\rho}{k}\in\Z_{> 0}.
\end{cases}
\end{align*}
Let us recall that $\big[\Psi(s)\big]_{(\omega+\e^{\i\theta_j\pls}\R_{\geq 0},\, \omega)}$ denotes the equivalence class of the ramified singular germ $\Psi(s)$ at $\omega$ (with ramification cut over $\omega+\e^{\i\theta_j\pls}\R_{\geq 0}$) modulo analytic germs at $\omega$ (see p.\pageref{p:germofhyperfunction}).\footnote{In the notation of \cite[p.200]{MitschiSauzin} the equivalent of $\big[\Psi(s)\big]_{(\omega+\e^{\i\theta_j\pls}\R_{\geq 0},\, \omega)}$ is \ $\textrm{sing}_\omega\big(\Psi(s)\big)$.}
This means that the Fourier coefficient $A'_{\omega,j}$ can be read from the coefficient of the leading singular term of $\Psi(s)$ at $\omega$.

In particular, in the special case $\rho=0$, the distribution $\Psi(s)$ takes at each singularity $\omega$ the simple form (called a \emph{simple resurgent singularity} in \cite{MitschiSauzin}):
\begin{equation*}\label{eq:srs}
\Psi(s)=\mfrac{{A}'_{\omega,j}}{2\pi\i }(s\!-\!\omega)^{-1} \mod \log(s\!-\!\omega)\,\C\{s\!-\!\omega\}+(s\!-\!\omega)^{\frac{1-k}{k}}\C\{(s\!-\!\omega)^{\frac{1}{k}} \}.
\end{equation*}

\section{Study of the orbit and proofs of main results}\label{sec:Thetaorbit}

\subsection{Reading the formal class from the orbit}
First of all we show that one can ``read'' any initial part of the Taylor expansion of $f$ at $0$, in particular the formal invariants of $f$, from asymptotics of the sequence $\orbit=(x_n)_{n\in\N}$. Similar statement is proven in \cite{Resman1} using the asymptotic expansion of the $\epsilon$-neighborhood of $\orbit$.

\begin{proposition}
	The map $\N\to\orbit\subseteq \mathbb C$, $m\mapsto x_m$, has a complete   asymptotic expansion into a power-log series, consisting of powers of $m^{-\frac1k}$ and of $\,\rho\frac{\log m}{m}$,
	\[x_m \sim\hat\psi_\orbit^{\circ(-1)}(m),\qquad\text{as }\ m\to+\infty,\]
	where $\hat\psi_\orbit^{\circ(-1)}(m)\in m^{-\frac1k}\C\llbracket m^{-\frac1k},\rho\frac{\log m}{m}\rrbracket$.
\end{proposition}

\begin{proof}
Since the sectorial Fatou coordinate $\psi_\orbit$ is asymptotic on $W_\orbit$ to $\hat\psi_\orbit(x)$ \eqref{eq:formalFatou}, 
then also $x_m=\psi_\orbit^{\circ(-1)}(m)$ is asymptotic to $\hat\psi_\orbit^{\circ(-1)}(m)$ as $m\to+\infty$. 
The formal inverse $\hat\psi_\orbit^{\circ(-1)}(t)$ is a formal transseries in $t^{-\frac{1}{k}}$ and $\rho t^{-1}\log t$ \cite[Theorem 3.1]{Kimura}, \cite[eq. (5.1.6)]{Ecalle3}.
\end{proof}

In particular, since $\psi_\orbit^{\circ(-1)}(m)=(-akm)^{-\frac1k}+o(m^{-\frac1k})$, the topological model \eqref{eq:fmod} can be read from $\orbit$ as:
\[k=\lim_{m\to+\infty}-\frac{\log m}{\log{x_m}},\qquad a=\lim_{m\to+\infty}-\frac{x_m^{-k}}{km}.\]

\subsection{Study of the dynamic theta function of an orbit.} ~\\
Let $\orbit=\{x_0,x_1,x_2,\ldots\}$ be a forward orbit \eqref{eq:orbit}, and 
$\Theta_\orbit(s)=\sum_{x_i\in\orbit}\e^{-s\,t(x_i)}$
its dynamic theta function \eqref{eq:Theta}.
Let $W_\orbit$ be the attractive petal of Theorem~\ref{thm:sectorialFatou} which contains the orbit $\orbit$ (see Figure~\ref{figure:domain1}), 
and let
\[\psi_\orbit(x)=\BSum_\alpha\{\hat\psi_\orbit\}(x)\] 
with $\alpha$ in some interval:
\begin{equation}\label{eq:thetaalpha}
	\alpha\in\ ]\,\underline{\theta},\overline{\theta}\,[\ = \ ]\tfrac{\pi}{2},\tfrac{3\pi}{2}[\mod 2\pi\Z,
\end{equation} 
be the unique sectorial Fatou coordinate for $f$ on $W_\orbit$ such that $\psi_\orbit(x_0)=0$, and therefore  
\[\psi_\orbit(x_n)=n\quad\text{for all $n\in\N$}.\]

In the expression \eqref{eq:Thetadelta} of $\Theta_\orbit$, the distribution
$\delta_\orbit(x)=\sum_{x_i\in\orbit}\delta_{x_i}(x)$ can be written as
\begin{equation*}\label{eq:Diracorbit}
    \delta_\orbit(x)\d x=\big(\delta_\N\circ\psi_\orbit(x)\big) \d \psi_\orbit(x),\qquad \delta_\N(t)=\sum_{n\in\N}\delta_n(t),
\end{equation*}
since the Diracs $\delta_n(t)$ transform in the way of the differential forms $\delta_n(t)\d t$.
By the residue theorem, we have
\begin{equation*}\label{eq:Bm}
	\delta_\N(t)=\BV[B_m(t)]_{\R_{\geq0}},\qquad  B_m(t)=\frac{\e^{2\pi\i  mt}}{\e^{2\pi\i  t}-1},
\qquad\text{for any $m\in\Z$.}
\end{equation*}
Therefore the formula \eqref{eq:Thetadelta} for $\Theta_\orbit(s)$ can be rephrased using the residue theorem as:
\begin{equation}\label{eq:ThetaB}
	\begin{aligned}
		\Theta_\orbit(s) &=\int_{\psi_\orbit(x)\in\Circ(\R_{\geq 0})}(B_m\!\circ\psi_\orbit)\,\e^{-s\,t(x)} \d\psi_\orbit(x),\qquad \re(s)>0,
	\end{aligned}
\end{equation}
for any $m\in\Z$.

\begin{proposition}\label{prop:Theta}
For any $\alpha\in\ ]\,\underline{\theta},\overline{\theta}\,[$ from \eqref{eq:thetaalpha}, the function	
$\Theta_\orbit(s)$ \eqref{eq:ThetaB} admits an analytic continuation from $\{\Re(s)>0\}$ to the slit plane $\C\sminus\bigcup_{\omega\in 2\pi\i \Z}(\omega+\e^{\i\alpha}\R_{\geq0})$, see Figure~\ref{figure:1}.	For each $\omega\in 2\pi\i \Z$,
	\begin{equation}\label{eq:BVThetaomega}
		\tfrac{1}{2\pi\i }\BV[\Theta_\orbit(s)]_{\omega+\e^{\i\alpha}\R_{\geq 0}}
		=\BV\big[\widetilde\Borel_\alpha\{\e^{\omega\hat\psi_\orbit}\tfrac{\d\hat\psi_\orbit}{\d t}\}(s)\big]_{\omega+\e^{\i\alpha}\R_{\geq 0}}
		=\Borel_\alpha\{\e^{\omega\hat\psi_\orbit}\tfrac{\d\hat\psi_\orbit}{\d t}\}(s).
	\end{equation}
Hence,
\begin{equation}\label{eq:BVThetaFatou0}
	\tfrac{1}{2\pi\i }\BV\big[\tfrac{\Theta_\orbit(s)}{s}\big]_{\omega+\e^{\i\alpha}\R_{\geq 0}}=   
	\begin{cases}
		\Borel_\alpha\{\hat\psi_\orbit\}(s),&\omega=0,\\[6pt] \tfrac{1}{\omega}\Borel_\alpha\{\e^{\omega\hat\psi_\orbit}\}(s),& \omega\in 2\pi\i\Z\sminus\{0\},
	\end{cases}
\end{equation}
and
\begin{equation}\label{eq:BVThetaFatou}
\tfrac{1}{2\pi\i }\int_{\Circ(\omega+\e^{\i\alpha}\R_{\geq 0})}\!\!\! \tfrac{\Theta_\orbit(s)}{s}\,\e^{s\,t(x)}\d s=
	\begin{cases}
	\psi_\orbit(x),&\omega=0,\\[6pt] \tfrac{1}{\omega} \e^{\omega \psi_\orbit(x)},& \omega\in 2\pi\i\Z\sminus\{0\}.
	\end{cases}
\end{equation}
\end{proposition}

\begin{figure}[t]
		\begin{tikzpicture}[yscale=.7, xscale=.9]
			\fill[fill=black!25] (-1.5,0) arc[start angle=0, end angle=80,radius=1.2] -- ++(170:2) -- ++(90:3) -- ++(0:9.3) |- cycle; 
			\fill[fill=black!25] (-1.5,0) arc[start angle=0, end angle=-80,radius=1.2] -- ++(-170:2) -- ++(-90:3) -- ++(0:9.3) |- cycle; 
			\draw[dashed, gray] (-1.5,0) arc[start angle=0, end angle=80,radius=1.2] -- ++(170:2);
			\draw[dashed, gray] (-1.5,0) arc[start angle=0, end angle=-80,radius=1.2] -- ++(-170:2);
			\foreach \n in {-4,...,4} {\filldraw (\n,0) circle (2pt);};
			\draw (0,0) node[right]{$0$};
			\draw[thick, ->] (4.5,0.4) -- (0,0.4) arc[start angle=90, end angle=270,radius=0.4] -- (4.5,-0.4);
			\draw[thick,red, ->] (-0.5,0) arc[start angle=180, end angle=120,radius=1] -- ++(30:4.5);
			\draw[red] (45:5) node[left]{$b+\e^{-\i \beta}\R_{\geq 0}$};
			\filldraw[red] (-0.5,0) circle (2pt); 
			\draw (-.7,-.4) node{$b$};
			\draw[thick,dashed, ->] (-1,4) -- (0,-4);
			\draw (0,-3.5) node[right]{$b+\i\e^{-\i \alpha}\R$};
			\draw (3.3,-0.5) node[below]{$\Circ(\R_{\geq 0})$};
		\end{tikzpicture}
\caption{ The image $\psi_\orbit(W_\orbit)$ of the petal $W_\orbit$, on which the sectorial Fatou coordinate $\psi_\orbit$ is defined, in the $\psi_\orbit(x)$ coordinate. 
The poles of $(B_m\!\circ\psi_\orbit)$ are at the points $\psi_\orbit\in\Z$. Note that we have chosen the constant in the Fatou coordinate $\psi_\orbit$ so that $\psi_\orbit(x_0)=0$, so $0$ is in the image $\psi_\orbit(W_\orbit).$}
	\label{figure:domain1}
\end{figure}
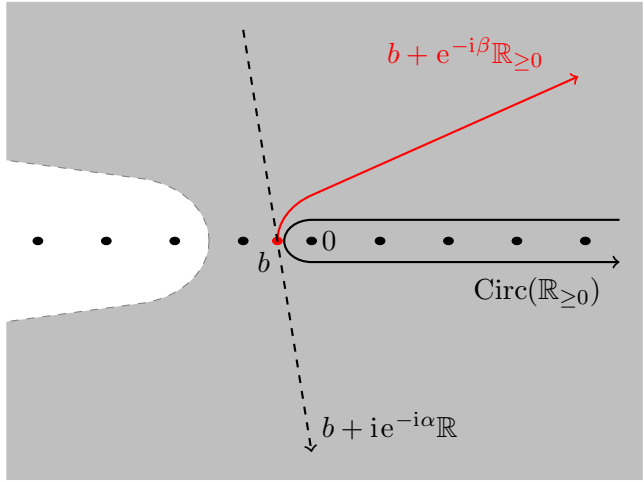 

Proposition~\ref{prop:Theta} gives a way to read the Birkhoff-\'Ecalle-Voronin invariants from the coefficients of leading singular terms at singularities of  $\Theta_\orbit(s)$. It states that the
differences of the coefficients of the leading singular terms of $\frac{\Theta_\orbit(s)}{2\pi\i\,s}$ at any $\omega\in 2\pi\i\Z^*$ on two consecutive sheets (with respect to the ramification at $s=0$) are precisely the leading coefficients of
$\Borel_\alpha\{\psi_\orbit\}(s)$, see Section~\ref{sec:3.3}.
Therefore, the Birkhoff--\'Ecalle--Voronin invariants can be read from the coefficients of the leading singular terms of $\Theta_\orbit(s)$ at the singularities $\omega\in2\pi\i \mathbb Z^*$ on the first $k+1$ sheets (with respect to the ramification at $s=0$).

\begin{proof}
By changing the integration contour in \eqref{eq:ThetaB} we extend $\Theta_\orbit(s)$ from $\Re(s)>0$ to the slit plane $\C\sminus\bigcup_{\omega\in 2\pi\i \Z}(\omega+\e^{\i\alpha}\R_{\geq0})$, for any $\alpha\in\ ]\,\underline{\theta},\overline{\theta}\,[$.
We deform the contour  $\psi_\orbit\in\Circ(\R_{\geq 0})$  to a line $\psi_\orbit\in b+\i\e^{-\i\alpha}\R$ with some $-1<b<0$ and any $\alpha\in\ ]\,\underline{\theta},\overline{\theta}\,[$  (see Figure~\ref{figure:domain1}),
and rewrite \eqref{eq:ThetaB} as
\begin{equation}\label{eq:ThetaC}
\Theta_\orbit(s)=\int_{b+\i \e^{-\i\alpha}\R}\!\!(B_m\!\circ\psi_\orbit(x))\,\e^{-s\,t(x)}\d\psi_\orbit(x),\qquad b\in\ ]\!-\!1,0[.
\end{equation}
While the original integral \eqref{eq:ThetaB} is convergent on the half-plane $\{\re(s)>0\}$, the deformed integral \eqref{eq:ThetaC} is
defined on the strip
$\big\{\frac{\im(\e^{-\i \alpha}s)}{\cos\alpha}\in\ ]2\pi(m\!-\!1),2\pi m[\,\big\}$ between the points $2\pi\i m$ and $2\pi\i(m\!-\!1)$. 
The strip of convergence of \eqref{eq:ThetaC} is deduced by looking at the exponential rate of growth of $(B_m\circ\psi_\orbit)\tfrac{\d\psi_\orbit}{\d t}$ at the two 
ends of the integration line (cf. p.\pageref{p:borel}):
\[(B_m\circ\psi_\orbit)\tfrac{\d\psi_\orbit}{\d t}\approx\begin{cases}
\e^{2\pi\i(m-1)\,t(x)}&\text{when }\ \psi_\orbit\to b+\i\e^{-\i\alpha}(+\infty),\\
\e^{2\pi\i m\,t(x)}&\text{when }\ \psi_\orbit\to b-\i\e^{-\i\alpha}(+\infty).
\end{cases}\]
By the residue theorem, the integrals defined by \eqref{eq:ThetaB} and by \eqref{eq:ThetaC}  agree on the half-strip $\{\Re(s)>0\}\cap \big\{\frac{\im(\e^{-\i \alpha}s)}{\cos\alpha}\in\ ]2\pi(m\!-\!1),2\pi m[\,\big\}$.
Since $m\in\mathbb Z$ can be chosen arbitrarily,
the function $\Theta_\orbit$ extends to the whole slit plane $\C\sminus\bigcup_{\omega\in 2\pi\i \Z}(\omega+\e^{\i\alpha}\R_{\geq0})$.

\smallskip

In order to show the formula \eqref{eq:BVThetaomega}, we first expand 
	\begin{equation}\label{eq:expo}
		B_m(t)=\begin{cases}
			-\!\!\displaystyle\sum_{n=m}^{+\infty}\e^{2\pi\i nt},& \im t>0,\\
			\,\displaystyle\sum^{m-1}_{n=-\infty}\e^{2\pi\i nt},& \im t<0,
		\end{cases}	
	\end{equation}
	as an absolutely convergent sum on the upper and the lower half-plane. 
Dividing the integration contour  
$\psi_\orbit\in\Circ(\R_{\geq 0})$  of \eqref{eq:ThetaB} into two parts  as $(b+\e^{\i0\mns}\R_{\geq 0})-(b+\e^{\i0\pls}\R_{\geq 0})$ with some $-1<b<0$, 
expanding $B_m\circ\psi_\orbit$ according to \eqref{eq:expo} in the upper and lower half-planes, and
exchanging the sum and the integral, we obtain 
\begin{equation}\label{eq:ThetaBsum}
\begin{aligned}
	\Theta_\orbit(s)
	&=\!\begin{multlined}[t][.7\displaywidth]
			\sum_{\frac{\omega}{2\pi\i}\in\Z_{\geq m}}\!\!\int_{ b+\e^{\i0\pls}\R_{\geq 0}}\!\!\!\!\!
			\e^{\omega\psi_\orbit(x)}\,\e^{-s\,t(x)}\d\psi_\orbit(x)\\
			+\sum_{\frac{\omega}{2\pi\i}\in \Z_{\leq m-1}}\!\!\int_{ b+\e^{\i0\mns}\R_{\geq 0}}\!\!\!\!\! 
			\e^{\omega\psi_\orbit(x)}\,\e^{-s\,t(x)}\d\psi_\orbit(x)
		\end{multlined}	\\[6pt]
	&=\sum_{\frac{\omega}{2\pi\i}\in\Z}\int_{b+\e^{-\i\beta}\R_{\geq 0}}\!\!\!\! \e^{\omega\psi_\orbit(x)}\,\e^{-s\,t(x)}\d\psi_\orbit(x),\\
	&=\sum_{\frac{\omega}{2\pi\i}\in\Z}2\pi\i\,\widetilde\Borel_\alpha\big\{\e^{\omega\psi_\orbit(x)}\,\tfrac{\d\psi_\orbit(x)}{\d t(x)}\big\}(s),\qquad\quad
			\beta\in\bigcup_{\alpha\in\ ]\,\underline{\theta}, \overline{\theta}\,[} ]\alpha\!-\!\tfrac{\pi}{2},\alpha\!+\!\tfrac{\pi}{2}[,
\end{aligned}
\end{equation}
where the direction of the integration ray $\psi_\orbit\in b+\e^{-\i\beta}\R_{\geq 0}$ of each term can now be varied within the whole petal $\psi_\orbit(W_\orbit)$. 
The ray $b+\e^{-\i\beta}\R_{\geq 0}$ may possibly need to be deformed  at its initial part in order to stay in $\psi_\orbit(W_\orbit)$.

\smallskip

We now show that each singularity $\omega\in 2\pi\i\Z$ of $\Theta_\orbit(s)$, as approached from  $\{\re(s)>0\}$, comes precisely from the corresponding term 
$2\pi\i\,\widetilde\Borel_\alpha\big\{\e^{\omega\psi_\orbit(x)}\,\tfrac{\d\psi_\orbit(x)}{\d t(x)}\big\}(s)$	in \eqref{eq:ThetaBsum}. That is, that the remainder of the sum does not have any singularity at $\omega$.

Let $\omega=2\pi\i m$, and fix a direction $\alpha\in\ ]\,\underline{\theta}, \overline{\theta}\,[$. 
The partial sum (with $\beta=\alpha+\tfrac\pi2$)
\[
\sum_{\frac{\omega}{2\pi\i}\in\Z_{\geq m+1}}\!\!\int_{b-\i\e^{-\i\alpha}\R_{\geq 0}}\!\!\!\!\!\!\! \e^{\omega\psi_\orbit(x)}\e^{-s\,t(x)}\!\d\psi_\orbit(x)
=
-\!\int_{b-\i\e^{-\i\alpha}\R_{\geq 0}}\!\!\!\!\!\!
(B_{m+1}\circ\psi_\orbit)\,\e^{-s\,t(x)}\!\d\psi_\orbit(x),
\]
converges on the half plane $\big\{\frac{\im(\e^{-\i \alpha}s)}{\cos\alpha}< 2\pi(m\!+\!1)\big\}$,
while the partial sum (with $\beta=\alpha+\tfrac{3\pi}{2}$)
\[\sum_{\frac{\omega}{2\pi\i}\in\Z_{\leq m-1}}\int_{b+\i\e^{-\i\alpha}\R_{\leq 0}}\!\!\!\!\!\! \e^{\omega\psi_\orbit(x)}\e^{-s\,t(x)}\!\d\psi_\orbit(x)
=\int_{b+\i\e^{-\i\alpha}\R_{\geq 0}}\!\!\!\!\! (B_m\circ\psi_\orbit)\,\e^{-s\,t(x)}\!\d\psi_\orbit(x),
\]
converges on the half-plane $\big\{\frac{\im(\e^{-\i \alpha}s)}{\cos\alpha}> 2\pi(m\!-\!1)\big\}$.
Therefore the sum of the two partial sums, which
is obtained from \eqref{eq:ThetaBsum} by omitting the $m$-th term, and therefore is equal to
\begin{equation}\label{eq:diffe}\Theta_\orbit(s)-\int_{b-\i\e^{-\i\alpha}\R_{\geq 0}}\!\!\! \e^{2\pi\i m \psi_\orbit(x)}\e^{-s\,t(x)}\d\psi_\orbit(x),\end{equation}
converges on the  $\alpha$-slanted strip $\big\{\frac{\im(\e^{-\i \alpha}s)}{\cos\alpha}\in\ ]2\pi(m\!-\!1),2\pi (m+1)[\,\big\}$. In particular, it is analytic at the point $s=2\pi\i m$.

Note that 
the integral
\[\int_{b-i\e^{-\i\alpha}\R_{\geq 0}}\!\!\! \e^{2\pi\i m \psi_\orbit(x)}\,\e^{-s\,t(x)}\d\psi_\orbit(x) \]
is defined  on the  half-plane $\big\{\frac{\im(\e^{-\i \alpha}s)}{\cos\alpha}<2\pi m\big\}$. However, it
can be extended analytically to the slit plane $\mathbb C\setminus\{2\pi\i m+\e^{i\alpha}\mathbb R_{\geq 0}\}$ by changing the line of integration to $b-\e^{-i\beta}\mathbb R_{\geq 0}$, \, $\beta\in(\alpha+\frac{\pi}{2},\alpha+\frac{3\pi}{2})$, as  in \eqref{eq:ThetaBsum}.

Hence, for $\omega=2\pi\i m$
\begin{equation*}\label{eq:hi}
\begin{aligned}
	\tfrac{1}{2\pi\i}\BV[\Theta_\orbit(s)]_{\omega+\e^{\i\alpha}\R_{\geq 0}}
	&=\tfrac{1}{2\pi\i}\BV\big[\!\int_{b+\e^{-\i\beta}\R_{\geq 0}}\!\!\!\!\! \e^{\omega\psi_\orbit(x)}\e^{-s\,t(x)}\d\psi_\orbit(x)\,\big]_{\omega+\e^{\i\alpha}\R_{\geq 0}}\\
	&=\BV\big[\widetilde\Borel_\alpha\{
	\e^{\omega\hat\psi_\orbit}\tfrac{\d\hat\psi_\orbit}{\d t}\}(s)\big]_{\omega+\e^{\i\alpha}\R_{\geq 0}}\\
	&=\Borel_\alpha\{
	\e^{\omega\hat\psi_\orbit}\tfrac{\d\hat\psi_\orbit}{\d t}\}(s).
\end{aligned}
\end{equation*}

\smallskip
Finally, formulas \eqref{eq:BVThetaFatou0} and \eqref{eq:BVThetaFatou} now follow from \eqref{eq:BVThetaomega} by standard properties of Borel and Laplace transforms.
\end{proof}

\begin{proposition}[Resurgence]\label{prop:resurgenceofTheta}
The function $\Theta_\orbit(s)$ extends  analytically  to $\widetilde{\C\sminus 2\pi\i\Z}$ and has at most exponential growth along any ray $c+\e^{\i\beta}\R_{\geq 0}$ for any point $c$ of the covering surface and any $\beta\notin \frac{\pi}{2}+\pi\Z$.
	
Furthermore, if  $\rho=0$, then $\Theta_\orbit(s)$ is a simple resurgent function on $\widetilde{\C\sminus 2\pi\i\Z}$, meaning that all its singularities
$\omega\in2\pi\i\Z$, when accessed from any sheet, have the form 
$(s\!-\!\omega)^{-1}\C\{(s\!-\!\omega)^{\frac{1}{k}}\}
+\log(s\!-\!\omega)\,\C\{s\!-\!\omega\}$. 
\end{proposition} 

\begin{proof}
We call the \emph{main sheet} the domain $\C\setminus\bigcup_{\omega\in 2\pi\i\Z}(\omega+\R_{\leq 0})$ containing the right half-plane $\{\Re s>0\}$ on which $\Theta_\orbit(s)$ is initially defined. 
The effect of continuing $\Theta_\orbit(s)$ across one of the cuts $\omega+\R_{\leq 0}$ to another sheet is by Proposition~\ref{prop:Theta} the same as taking the value of $\Theta_\orbit(s)$ on the main sheet and adding to it (or subtracting from it, depending on the direction) $2\pi\i\,\Borel_\alpha\{\e^{\omega\psi_\orbit}\tfrac{\d\psi_\orbit}{\d t}\}(s)$.
By iterating this procedure one can express the value of $\Theta_\orbit(s)$ on any sheet as a sum of its value on the main sheet and of finitely many various analytic continuations of $\mp2\pi\i\,\Borel_{\alpha}\{\e^{\omega\psi_\orbit}\tfrac{\d\psi_\orbit}{\d t}\}$ over the  $\omega$'s encircled by the path along which $\Theta_\orbit$ is continued.
By Theorem~\ref{thm:Ecalle} we know that $\Borel_\alpha\{\psi_\orbit\}(s)$ is resurgent (resp. simple resurgent if $\rho=0$), and therefore so is each
\[2\pi\i\Borel_\alpha\{\e^{\omega\psi_\orbit}\tfrac{\d\psi_\orbit}{\d t}\}(s)=\begin{cases}
2\pi\i\,s\Borel_\alpha\{\psi_\orbit\}(s)&\omega=0,\\
\tfrac{2\pi\i\,s}{\omega}\Borel_\alpha\{\e^{\omega\psi_\orbit}\}(s),&\omega\in2\pi\i\Z\setminus\{0\},
\end{cases}\]
as it is a minor Borel transform of a composition of $\psi_\orbit$ into an analytic germ \cite[Theorem 6.32]{MitschiSauzin}.
\end{proof}

\subsection{Alternative approach to Theorems~\ref{thm:1} and~\ref{cor:1}} \label{sec:treci}\

In this section, we sketch alternative proofs of Theorems~\ref{thm:1} and~\ref{cor:1} in case $\rho=0$. 
In Proposition~\ref{prop:druga} we express the dynamic theta function $\Theta_\orbit$ as a convolution involving the Borel transform of the inverse Fatou coordinate of $\orbit$. In this way we are able to relate the coefficients of the leading singular terms of singularities of the theta function with  those of the Borel transform of the inverse Fatou coordinate, which are known to correspond to the Birkhoff--\'Ecalle--Voronin invariants. This approach is motivated by the similar approach in \cite{Candelpergher, Sauz06, MitschiSauzin}.

\smallskip

Let $\psi_\orbit(x)=t(x)+\rho\log x+C+O(x)$ be the sectorial Fatou coordinate, such that $\psi_\orbit(x_0)=0$, and let
\begin{equation}\label{eq:u} 
u(x):=t\circ\psi_\orbit^{\circ(-1)}\!\circ t(x)-t(x)=-\rho\log x-C+O(x,\rho x^k\log x).
\end{equation}
Let $\alpha\in\ ]\,\underline{\theta},\overline{\theta}\,[$ be a corresponding direction \eqref{eq:thetaalpha} of Borel summation on the petal $W_\orbit$, $\psi_\orbit=\BSum_\alpha\{\hat\psi_\orbit\}$.
For $\sigma\in\mathbb C$ we define:
\begin{equation}\label{eq:Q}
\begin{aligned}
&\tilde Q(s,\sigma):=\widetilde\Borel_{\alpha}\{\e^{-\sigma u(x)}\}(s),\\
&Q(s,\sigma):=\Borel_{\alpha}\{\e^{-\sigma u(x)}\}(s)=\BV[\tilde Q(s,\sigma)]_{s\in \e^{\i\alpha}\R_{\geq 0}}.
\end{aligned}
\end{equation}
The function $\tilde Q(s,\sigma)$ is defined and analytic on
$(s,\sigma)\in (\C\sminus \e^{\i\alpha}\R_{\geq 0})\times\C$ and extends analytically to $\widetilde{(\C\sminus 2\pi\i\Z)}\times\C$, with at most exponential growth in $s$ along any non-vertical ray in $\widetilde{\C\sminus 2\pi\i\Z}$. These properties of $\tilde Q(s,\sigma)$ follow from knowing that $\widetilde\Borel_\alpha\{u(x)\}(s)$ has those properties \cite{Ecalle1}, \cite[Theorem 7.6]{MitschiSauzin}, 
and that they are preserved under composition with the entire function $u\mapsto \e^{-\sigma u}$, $\sigma\in\mathbb C$ \cite[Theorem 6.32]{MitschiSauzin}.

\begin{proposition}\label{prop:druga}
 For $\alpha\in\ ]\,\underline{\theta},\overline{\theta}\,[$ \eqref{eq:thetaalpha},
and $s\in\C\sminus\bigcup_{\omega\in 2\pi\i \Z}(\omega+\e^{\i\alpha}\R_{\geq 0})$,
\begin{equation}\label{eq:ThetaQ}
	\Theta_\orbit(s)=\int_{\Circ(\e^{\i\alpha}\R_{\geq 0})} \frac{\tilde Q(\xi,s)}{1-\e^{\,\xi-s}}\d\xi
	=\int_{\e^{\i\alpha}\R_{\geq 0}} \frac{Q(\xi,s)}{1-\e^{\,\xi-s}}\d\xi,
\end{equation}	
where the functions $Q(s,\sigma)$ and $\tilde Q(s,\sigma)$ are defined in \eqref{eq:Q}.
\end{proposition}

\begin{proof}
Using expression \eqref{eq:ThetaC}, and changing the variable of the integration, we get:
\begin{equation}\label{eq:ThetaD}
\begin{aligned}
	\Theta_\orbit(s)&=\int_{b+\i \e^{-\i\alpha}\R}\!(B_m\!\circ\psi_\orbit(x))\,\e^{-s\,t(x)}\d\psi_\orbit(x)\\
	&=\int_{b+\i \e^{-\i\alpha}\R}\!(B_m\!\circ t(x))\,\e^{-s\,t\circ\psi_\orbit^{\circ(-1)}\!\circ t(x)}\d t(x)\\
	&=\int_{b+\i \e^{-\i\alpha}\R}\!(B_m\!\circ t(x))\,\e^{-su(x)}\e^{-s\,t(x)}\d t(x),
\end{aligned}
\end{equation}
for $s$ in the strip
\begin{equation}\label{eq:strip}
\Big\{\tfrac{\im(\e^{-\i \alpha}s)}{\cos\alpha}\in\ ]2\pi(m\!-\!1),\, 2\pi m[\,\Big\}\end{equation} 
between the points $2\pi\i m$ and $2\pi\i(m-1)$, parallel to $\e^{\i\alpha}\R$ (see the discussion after \eqref{eq:ThetaC}).

We want to rewrite the integral \eqref{eq:ThetaD} as a convolution of two Borel transforms. In order to assure convergence, we shall divide the integration line $b+\i\e^{-\i\alpha}\R$ into two rays 
$b+\i\e^{-\i\alpha}\R_{\geq 0}$ and $b-\i \e^{-\i\alpha}\R_{\geq 0}$. 
We denote
\[\tilde\Phi_\pm(s,\sigma):=\int_{b\mp\i \e^{-\i\alpha}\R_{\geq 0}}\!\!\!(B_m\!\circ t(x))\,\e^{-\sigma u(x)}\e^{-s\,t(x)}\d t(x),\]
and $\Phi(s,\sigma):=\tilde\Phi_-(s,\sigma)-\tilde\Phi_+(s,\sigma)=\int_{b+\i\e^{-\i\alpha}\R}(B_m\circ t)\,\e^{-\sigma u}\e^{-st}\d t$, so that $\Theta_\orbit(s)=\Phi(s,s)$.
We rewrite each $\tilde\Phi_\pm(s,\sigma)$ as an integral over the whole line $b+\i\e^{-\i\alpha}\R$
\[\tilde\Phi_\pm(s,\sigma)=\int_{b+\i\e^{-\i\alpha}\R}\!
g_\pm(x,\sigma)\cdot h(x)\,\e^{-s\,t(x)}\d t(x)=
2\pi\i\,\Borel_\alpha\big\{g_\pm\cdot h\big\}(s),
\]
where
$g_\pm(x,\sigma):= (\mathbbm{1}_{b\mp\i\e^{-\i\alpha}\R_{\geq 0} } \circ t(x))\cdot\e^{-\sigma u(x)}$ and $h(x):=B_m\!\circ t(x)$.
Now we apply the standard convolution identity
for two-sided Borel--Laplace--Fourier type transforms of a product of two functions on the line $t(x)\in b+\i\e^{-\i\alpha}\R$: 
\[\tilde \Phi_\pm(s,\sigma)=2\pi\i\,\Borel_\alpha\big\{g_\pm \cdot h\big\}(s)
=2\pi\i\int_{\i\e^{\i\alpha}(0\pms)+\e^{\i\alpha}\R}\!\!\!\!
\Borel_\alpha\big\{g_\pm\big\}(\xi)\cdot\Borel_\alpha\big\{h\big\}(s-\xi)\d\xi,\]
for $s$ in the strip \eqref{eq:strip}.
Here,
\begin{align*}
2\pi\i\,\Borel_\alpha\{h\}(s)&=2\pi\i\,\Borel_\alpha\{B_m\!\circ t(x)\}(s)=
\int_{b+\i\e^{-\i\alpha}\R} \tfrac{\e^{2\pi\i m\,t(x)}}{\e^{2\pi\i\, t(x)}-1}\e^{-s\,t(x)}\d t(x) =\mfrac{1}{1-\e^{-s}}
\end{align*}
by the residue theorem, on the strip \eqref{eq:strip}, and
\[\Borel_\alpha\{g_\pm\}(s)= \tfrac{1}{2\pi\i}\int_{b\mp\i\e^{-\i\alpha}\R_{\geq 0}}\!\!\!\! \e^{-\sigma u(x)}\e^{-s\,t(x)}\d t(x) = \tilde Q(s,\sigma), \]
on the half-plane 
$\big\{\Re(\e^{-\i\beta_\pm}s)>0\big\}$ bounded by the line $\e^{\i\alpha}\R$,
where $\beta_+=\alpha+\frac\pi2$, $\beta_-=\alpha+\frac{3\pi}{2}$. 
We have
\[\Borel_\alpha\{g_-\}(s)-\Borel_\alpha\{g_+\}(s)=\BV[\tilde Q(s,\sigma)]_{\e^{\i\alpha}\R}=
\begin{cases}
\BV[\tilde Q(s,\sigma)]_{\e^{\i\alpha}\R_{\geq0}},\!\!\! & s\in\e^{\i\alpha}\R_{\geq 0},\\0,&s\in\e^{\i\alpha}\R_{<0},
\end{cases}
\]
in distributional sense, hence, on the strip \eqref{eq:strip}:
\begin{equation}\label{eq:PHI}
\Phi(s,\sigma)=\tilde\Phi_-(s,\sigma)-\tilde\Phi_+(s,\sigma)=\int_{\Circ(\e^{\i\alpha}\R_{\geq 0})} \frac{\tilde Q(\xi,\sigma)}{1-\e^{\,\xi-s}}\d\xi.
\end{equation}
Since  the identity \eqref{eq:PHI} is independent of the integer $m\in\Z$ determining the strip \eqref{eq:strip}, it is true on the whole slit plane $s\in\C\sminus\bigcup_{\omega\in 2\pi\i \Z}(\omega+\e^{\i\alpha}\R_{\geq 0})$, and for all $\sigma\in\C$.
The identity \eqref{eq:ThetaQ} now follows by restricting to $\{\sigma=s\}$. 
\end{proof}

In particular, if the residual invariant $\rho$ \eqref{eq:resit} is zero, then 
$u(x)=-C+O(x)$ \eqref{eq:u} and we write:
\begin{equation}\label{eq:QBu}
Q(s,\sigma)=\e^{\sigma C}\delta_0(s)+ \Borel_{\alpha}\{\e^{-\sigma u(x)}\!-\e^{\sigma C}\}(s), \qquad \sigma\in\mathbb C,  
\end{equation}
where $\Borel_{\alpha}\{\e^{-\sigma u(x)}\!-\e^{\sigma C}\}(s)$ is in this case not just a distribution, but an actual analytic germ, which extends to a simple resurgent function on $\widetilde{\mathbb C\setminus 2\pi\i \mathbb Z}$. Indeed, it is a convolution exponential of a simple resurgent function $\Borel_{\alpha}(u(x))(s)$, see \cite[Proposition 3, \S 1]{MitschiSauzin} for details.
We now have $\Theta_\orbit(s)=\Phi(s,s)$, for 
\begin{equation}\label{eq:I}
\Phi(s,\sigma)=\int_{\e^{\i\alpha}\R_{\geq 0}}\! \frac{Q(\xi,\sigma)}{1-\e^{\,\xi-s}}\d\xi=
\frac{\e^{\sigma C}}{1-\e^{-s}}+\int_{\e^{\i\alpha}\R_{\geq 0}}\!\!\!\! \frac{\Borel_{\alpha}\{\e^{-\sigma u(x)}\!-\e^{\sigma C}\}(\xi)}{1-\e^{\,\xi-s}}\d\xi.
\end{equation}

\smallskip

\begin{proof}[Alternative proof of Propositions~\ref{prop:Theta} and ~\ref{prop:resurgenceofTheta}]
For simplicity let us assume that $\rho=0$ (the general case would be similar except for having to work with $\tilde Q(s,\sigma)$ \eqref{eq:Q} rather than $Q(s,\sigma)$ \eqref{eq:QBu}).

By \cite[Proposition 9b1]{Ecalle2}, \cite[Theorem 7.6]{MitschiSauzin},  $\Borel_\alpha\{u(x)\}(s)$ is simple resurgent when $\rho=0$, and therefore by \cite[Theorem 6.32]{MitschiSauzin} so is the function $\Borel_{\alpha}\{\e^{-\sigma u(x)}\!-\e^{\sigma C}\}(s)$ for all $\sigma\in\C$.

It is sufficient to verify that the singularities of $\Phi(s,\sigma)$ \eqref{eq:I} on the main sheet are of simple resurgent type. Indeed, by the residue theorem applied to \eqref{eq:I},
\begin{equation}\label{eq:ihi}
\BV[\Phi(s,\sigma)]_{\omega+\e^{\i\alpha}\R_{\geq 0}}=
2\pi\i\,\e^{\sigma C}\delta_\omega(s) + 
2\pi\i\,\Borel_\alpha\{\e^{-\sigma u(x)}-\e^{\sigma C}\}(s-\omega),
\end{equation}
where the second term belongs to $(s-\omega)^{\frac{1-k}{k}}\mathbb C\{(s-\omega)^{\frac1k}\}$ for any $\sigma\in\C$.
Therefore, the function $\Theta_\orbit(s)$ has locally at $\omega\in 2\pi\i\mathbb Z$ simple resurgent singularities. 

By \eqref{eq:ihi}, the analytic continuation of $\Phi(s,\sigma)$ from the main sheet $s\in\C\sminus\bigcup_{\omega\in 2\pi\i \Z}(\omega+\e^{\i\alpha}\R_{\geq 0})$ 
across one of the cuts $\omega+\e^{\i\alpha}\R_{\geq 0}$ differs from the corresponding value of $\Phi(s,\sigma)$ on the main sheet by  $\pm2\pi\i\,\Borel_\alpha\{\e^{-\sigma u(x)}-\e^{\sigma C}\}(s-\omega)$.
Iterating this procedure one can express the value of $\tilde\Phi(s,\sigma)$ on any sheet in terms of its value on the main sheet and a sum of finitely many analytic
continuations of the simple resurgent functions $\mp2\pi\i\,\Borel_{\alpha}\{\e^{-\sigma u(x)}\!-\e^{\sigma C}\}(s-\omega)$ over the encircled $\omega$'s. This proves Proposition~\ref{prop:resurgenceofTheta}.

The identity \eqref{eq:ihi} applied to $\BV[\Theta_\orbit(s)]_{\omega+\e^{\i\alpha}\R_{\geq 0}}=\BV[\Phi(s,s)]_{ \omega+\e^{\i\alpha}\R_{\geq 0}}$ can be also rewritten as:
\begin{align*}
 \BV[\Theta_\orbit(s)]_{\omega+\e^{\i\alpha}\R_{\geq 0}}&=
2\pi\i\,\Borel_\alpha\{\e^{-su(x)}\}(s-\omega)=
\int_{b+\i\e^{-\i\alpha}\R}\!\!\!\e^{-s\,u(x)}\e^{-(s-\omega)\,t(x)}\d t(x)\\
&=\int_{b+\i\e^{-\i\alpha}\R}\!\!\!\e^{\omega\psi_\orbit(x)}\e^{-s\,t(x)}\d \psi_\orbit(x)
=2\pi\i\,\Borel_\alpha\{\e^{\omega\psi_\orbit(x)}\tfrac{\d\psi_\orbit(x)}{\d t(x)}\}(s),   
\end{align*}
by applying an opposite change of coordinate of \eqref{eq:ThetaD}, 
which reproves Proposition~\ref{prop:Theta}.
\end{proof}

\let\oldbibitem\bibitem \renewcommand{\bibitem}[2][]{\oldbibitem{#2}} 

\end{document}